\documentclass{llncs}
\usepackage[english]{babel}

\usepackage[T1]{fontenc}
\usepackage{CJKutf8}

\usepackage{amsmath}
\usepackage{amssymb}
\usepackage{graphicx}
\usepackage[colorlinks=true, allcolors=blue]{hyperref}

\usepackage{enumitem}

\newcommand{\Gr}{\mathbf{Gr}}
\newcommand{\Ax}{\mathbf{Ax}}
\newcommand{\AxL}{\mathbf{AxL}}
\newcommand{\AxC}{\mathbf{AxC}}
\newcommand{\OrdL}{\mathbf{OrdL}}
\newcommand{\OrdB}{\mathbf{OrdB}}
\newcommand{\OrdC}{\mathbf{OrdCL}}
\newcommand{\OrdCB}{\mathbf{OrdCB}}
\newcommand{\AxPC}{\mathbf{AxPL}}
\newcommand{\PB}{\mathbf{PB}}
\newcommand{\AxPCL}{\mathbf{AxPCL}}
\newcommand{\PCB}{\mathbf{PCB}}

\newcommand{\CC}{\mathbf{CC}}
\newcommand{\Isolated}{\mathbf{Isolated}}

\title{Automated reasoning for proving non-orderability of groups}

\author{Alexei Lisitsa\inst{1}
\and
Zipei Nie\inst{2} \and
Alexei Vernitski\inst{3}
}
\institute{University of Liverpool,
  Liverpool, UK \\
\texttt{a.lisitsa@liverpool.ac.uk}\and
  Lagrange Mathematics and Computing Research Center, Huawei\\
\texttt{niezipei@huawei.com}\and 
University of Essex, Essex, UK\\
\texttt{asvern@essex.ac.uk}
}

\begin{document}

\maketitle

\begin{abstract}
We demonstrate how a generic automated theorem prover can be applied to establish the non-orderability of groups. Our approach incorporates various tools such as positive cones, torsions, generalised torsions and cofinal elements.
\end{abstract}

\section{Introduction}
The study of orderable groups has a rich history \cite{kokorin1974fully,mura1977orderable,kopytov1996right,glass1999partially} among group theorists. In the last decade of the 20th century, researchers gradually recognised the significance of orderability in topology, and it has since remained an active and thriving area of research. 

The three most extensively studied variants of group orderability are left order, bi-order, and (left-invariant) circular order. In this paper, we will also study the existence of bi-invariant circular orders. We consider the following algorithmic problem. 

\begin{problem}\label{Problem}
    Given a group $G$ with presentation $\langle S|R \rangle$ and a type of group orderability, determine whether $G$ is orderable.
\end{problem}

In general, Problem~\ref{Problem} has been proven to be undecidable for any type of orderability by \cite[Theorem 3.3]{baumslag2012algorithms}. Despite its undecidability, the problem remains highly intriguing when considering specific groups that arise in topology. Notably, the braid groups \cite{short1999orderings,rolfsen1998braids}, the mapping class groups \cite{rourke2000order,hyde2019group}, the fundamental groups of $3$-manifolds \cite{boyer2005orderable,boyer2013spaces,juhasz2015survey,ito2016alexander}, and lattices in Lie groups \cite{witte1994arithmetic} are among the groups where Problem~\ref{Problem} holds great significance.

We focus on the non-orderability aspect of Problem~\ref{Problem}, which involves finding a contradiction assuming the existence of an order. We believe that, by providing simple proofs demonstrating the non-orderability of a group of interest, one may deepen our understanding and reveal intriguing topological structures.

Most previous automated proofs of non-orderability are variants of the algorithm described in \cite[Section 8]{calegari2003laminations}. This algorithm relies on a short-lex automatic structure to tackle the word problem and seeks to identify contradictions within the positive cone of a left order. Later, Dunfield \cite{dunfield2020floer} enhanced this algorithm for the fundamental group of a finite-volume hyperbolic $3$-manifold by solving the word problem through an $SL_2(\mathbb{C})$ representation.

We present a methodology for establishing non-orderability using generic automated theorem proving instead of specialized algorithms. In contrast to previous approaches, our method offers a unified framework capable of handling all variants of orderability without any assumptions about the group. The flexibility of the imposed assumptions makes it easier to discover new proofs and results in non-orderability.

We provide many examples to illustrate our methodology, ranging in difficulty. A particularly interesting one is Example~\ref{ex:Hyde}, where we provide an alternative proof of the non-left-orderability \cite{hyde2019group} of $\mathrm{Homeo}(D,\partial D)$, the group of homeomorphisms of the disk fixing the boundary, using the concept of left absolutely cofinal elements; see also \cite{triestino2021james}.

\subsection{Main results and organisation of the paper}
In Section~\ref{sec:FP}, we present the first principles approach. As its name suggests, we directly break down the original problem into axioms in first-order logic, and then prove the non-orderability through the automatically-derived contradiction. To derive sufficient axioms from the group presentation, a necessary technical step is to identify inequalities using an automated finite model finder.

When compared to specific algorithms, like the one described in \cite[Section 8]{calegari2003laminations}, our approach offers a significant advantage: the ability to readily modify input axioms. We have noticed that the axiom of connectedness consumes excessive computational resources. Therefore, in Subsection~\ref{subsec:weakened}, we introduce the weakened theory approach, in which we substitute the axiom of connectedness with weaker assumptions.

In Subsection~\ref{subsec:positive_cone}, we discuss a further optimisation of our approach: the positive cone translation. This method has been used in the literature for establishing non-left-orderability and non-bi-orderability. We extend its applicability to left-invariant or bi-invariant circular orders. From the perspective of automated reasoning, the positive cone translation results in an equiconsistent theory with predicates of smaller arities, thereby enhancing the efficiency of our reasoning process.

In Section~\ref{sec:torsions}, we establish the equivalence between an element $g\in G$ being a torsion or a generalised torsion and the inconsistency of the corresponding weakened theory with respect to the pair $(e,g)$. One may compare this result with the well-known fact that the existence of a nontrivial torsion (resp. generalised torsion) implies non-left-orderability (resp. non-bi-orderability).

Furthermore, we extend this analysis to bi-invariant circular orders. If the axiom of cyclicity is dropped, and the axiom of connectedness is weakened with respect to the triple $(e, fgf^{-1}, g)$, then the theory for bi-invariant circular orders is inconsistent if and only if $f^{-1}$ is in the monoid generated by $f$ and the centraliser of $g$.

In Section~\ref{sec:convexity}, we further study the strength of the axiom of cyclicity for bi-invariant circular orders. For this purpose, we develop the theory of left relatively convex subgroups defined by Antol{\i}n, Dicks, and Sunic \cite{antolin2015left}. Unlike previous works such as \cite{kopytov1996right}, their definition does not impose restrictions on the left-orderability of the ambient group. We extend some properties of relatively convex subgroups of left-orderable groups to general cases.
\begin{enumerate}[label=(\alph*)]
    \item The equivalence (a) $\Leftrightarrow$ (c) in Proposition~\ref{prop-rel-convex-1} generalises the usual definition of relatively convex subgroups of left-orderable groups.
    \item The equivalence (a) $\Leftrightarrow$ (b) in Proposition~\ref{prop-rel-convex-2} generalises \cite[Theorem 1.4.10]{clay2010space}.
    \item Proposition~\ref{convex-intersection} generalises \cite[Proposition 5.1.10]{kopytov1996right}, establishing the closure property under arbitrary intersection.
\end{enumerate} 
Using the theory of left relatively convex subgroups, we demonstrate that a relation on the group $G$ satisfying all axioms except cyclicity exists if and only if the centraliser of each subset of $G$ is left relatively convex in $G$.

In Section~\ref{sec:cofinality}, we introduce the concept of the left relatively convex subgroup closure, denoted as $\mathrm{cl}(A)$, for a subset $A$. This closure is defined as the intersection of all left relatively convex subgroups containing $A$. Additionally, we define the left absolute cofinal subgroup as a subgroup $H$ for which $\mathrm{cl}(H)=G$. By the equivalence (a) $\Leftrightarrow$ (c) established in Proposition~\ref{prop-equiv-cofinal}, a subgroup is left absolutely cofinal if and only if it is cofinal with respect to every left total preorder. 

Based on the search for left absolute cofinal cyclic subgroups, we introduce a methodology for establishing the non-left-orderability of every nontrivial quotient using automated theorem proving. Our approach formalises a common practice in establishing the non-left-orderability, which involves tracking fixed points in the dynamical realisation. We believe this formalisation could also be used in other scenarios, including the development of fast algorithms such as the one detailed in \cite[Section 8]{calegari2003laminations}, as well as in human proofs.

Propositions~\ref{prop:main}, \ref{prop:weak}, \ref{weak-cbo}, \ref{symmetry}, \ref{prop:positive_cone}, \ref{prop:positive_cone_C}, \ref{prop-torsion}, \ref{prop-generalised-torsion}, \ref{prop-CBO-monoid}, \ref{prop-cl}, \ref{prop-fixed-point} establish the group theoretic implications of the consistency or inconsistency of various theories. The utilisation of an automated theorem prover like Prover9 enables the automatic detection of contradictions in a first-order theory, thereby implying corresponding group theoretic properties.

We provide numerous examples at the end of each section to illustrate our methodology. In Section~\ref{sec:tasks}, we elaborate our approach for performing automated reasoning tasks in these examples using Prover9 and Mace4. Furthermore, we provide a summary of the performance of these automated theorem provers when executing these tasks.

\section{Non-orderability from first principles} \label{sec:FP}

In this section, we demonstrate a methodology for reducing the non-orderability of a group $G$ with presentation $\langle S|R\rangle$ to automated theorem proving tasks. Specifically, we address the following variant of Problem~\ref{Problem}.
\begin{problem}\label{Problem'}
    Given a group $G$ with presentation $\langle S|R \rangle$ and a type of group orderability, provide an automated proof of the non-orderability of $G$.
\end{problem}

To achieve this goal, we first extract the assumptions of our problem setting in first-order logic. Then we utilise two tools for experimental evaluation: Prover9 and Mace4 \cite{prover9-mace4}. Prover9 is a widely used automated theorem prover that operates on first-order logic. It takes a set of logical axioms as input and attempts to find a proof of a conjecture. And the finite model finder Mace4 is designed to search for finite models that satisfy a given set of first-order logical formulas.

\subsection{Group axioms}\label{subsec:group_axioms}
A group is a set together with an associative binary operation $\cdot$, such that there exists an identity element, and every element has an inverse. In a first-order logic perspective, we consider the groups as models for the following standard system of axioms $\Gr$ in a vocabulary consisting a binary functional symbol $\cdot$ for group multiplication, a unary functional symbol $'$ (in postfix notation) for group inverse operation, and a constant $e$ for the identity element in the group: 
\begin{enumerate}[label=(\alph*)]
    \item $\forall x \forall y \forall z((x \cdot y) \cdot z = x \cdot (y \cdot z))$,\hfill\emph{(associativity)} 
    \item $\forall x(x \cdot e = e \cdot x = x)$,\hfill\emph{(identity element)} 
    \item $\forall x(x'\cdot x = x \cdot x' = e)$.\hfill\emph{(inverse element)} 
\end{enumerate}

For the group with presentation $\langle S|R\rangle$, we encode every generator in $S$ as an additional constant, and every relation in $R$ as an additional equational axiom, in a standard way. We denote this system of axioms by $\Ax_R$.

In addition to the group $G$ presented by $\langle S|R\rangle$, any quotient group of $G$ is also a model of $\Gr \cup \Ax_R$. To establish non-orderability, it is crucial to distinguish between $G$ and its quotient groups. Therefore, we need to include an additional situation-dependent set $\mathfrak{S}$ of true statements for the group $G$. In practice, $\mathfrak{S}$ can be taken as a set of inequalities in $G$.

\subsection{Order axioms}\label{subsec:order_axioms}
By definition, a \emph{left order} on a group $G$ is a linear order $<$ on $G$ that is invariant under left multiplication, and a \emph{bi-order} is one that is invariant under left and right multiplication. Thus the left order and bi-order satisfy the following axioms $\AxL$ for linear orders:
\begin{enumerate} [label=(\alph*)]
\item $\forall x (\neg (x < x))$,  \hfill\emph{(irreflexivity)} 
\item $\forall x \forall y \forall z((x < y) \land  (y < z) \to  (x < z))$, \hfill\emph{(transitivity)} 
\item $\forall x \forall y ((x=y) \lor (x < y)\lor(y<x))$. \hfill\emph{(connectedness)} 
\end{enumerate}

In addition, the left order satisfies the left-invariance axiom $\OrdL$:
\begin{enumerate} [label=(\alph*)]
\item[(d)] $\forall x \forall y \forall z ((x<y) \to  (z\cdot x < z\cdot y))$. \hfill\emph{(left-invariance)} 
\end{enumerate}
And the bi-order satisfies the bi-invariance axiom $\OrdB$:
\begin{enumerate} [label=(\alph*)]
\item[(e)] $\forall x \forall y \forall z \forall u ((x<y) \to  ((z\cdot x)\cdot u < (z\cdot y)\cdot u))$. \hfill\emph{(bi-invariance)} 
\end{enumerate}

By definition, a \emph{cyclic order} on a set is a ternary relation $C(\cdot,\cdot,\cdot)$ satisfying the following axioms $\AxC$:
\begin{enumerate} [label=(\alph*)]
\item $\forall x \forall y \forall z(C(x,y,z) \to  C(y,z,x))$, \hfill\emph{(cyclicity)} 
\item $\forall x \forall y (\neg C(x, y, y))$,  \hfill\emph{(irreflexivity)} 
\item $\forall x \forall y\forall z \forall u (C(x,y,z)\land C(x,z,u) \to C(x,y,u))$,  \hfill\emph{(transitivity)} 
\item $\forall x \forall y\forall z ((x=y)\lor(y=z)\lor(z=x)\lor C(x,y,z)\lor C(x,z,y)) $. \hfill\emph{(connectedness)} 
\end{enumerate}

When referring to a \emph{circular order} on a group $G$, it is conventionally understood as a cyclic order on the elements of $G$ that is invariant under left multiplication. So the circular order also satisfies the left-invariance axiom $\OrdC$:
\begin{enumerate} [label=(\alph*)]
\item[(e)] $\forall x \forall y \forall z\forall u (C(x,y,z) \to  C(u\cdot x ,u\cdot y,u\cdot z))$. \hfill\emph{(left-invariance)} 
\end{enumerate}
And the \emph{bi-invariant circular order} satisfies the bi-invariance axiom $\OrdCB$:
\begin{enumerate} [label=(\alph*)]
\item[(f)] $\forall x \forall y \forall z \forall u\forall v (C(x,y,z) \to  C((u\cdot x)\cdot v ,(u\cdot y)\cdot v,(u\cdot z)\cdot v))$. \hfill\emph{(bi-invariance)} 
\end{enumerate}

\subsection{The first principles approach}
If a group is orderable, then with the corresponding order it would constitute a model for the first-order theory with applicable axioms from Subsection~\ref{subsec:group_axioms} and Subsection~\ref{subsec:order_axioms}, which in turn would entail that the theory is consistent. In other words, we have the following proposition.

\begin{proposition} \label{prop:main}
    Let $G$ be a group with presentation $\langle S|R\rangle$. Let $\mathfrak{S}$ be a set of true statements for $G$. Then the following statements hold:
    \begin{enumerate}[label=(\alph*)]
        \item If $\Gr \cup \Ax_R\cup \AxL\cup \OrdL \cup \mathfrak{S}$ is inconsistent, then $G$ is not left-orderable.
        \item If $\Gr \cup \Ax_R\cup\AxL \cup \OrdB \cup \mathfrak{S}$ is inconsistent, then $G$ is not bi-orderable.
        \item If $\Gr \cup \Ax_R\cup \AxC \cup \OrdC \cup \mathfrak{S}$ is inconsistent, then $G$ is not circularly orderable.
        \item If $\Gr \cup \Ax_R\cup\AxC \cup \OrdCB \cup \mathfrak{S}$ is inconsistent, then $G$ does not admit a bi-invariant circular order.
    \end{enumerate}
\end{proposition}

According to Proposition~\ref{prop:main}, given the group presentation, to establish non-orderability of the presented group, one can apply an automated theorem prover in the first-order logic to derive a contradiction from the corresponding theory. 

Let us outline some important observations:
 \begin{enumerate}[label=(\alph*)]
    \item The proposed approach is not fully automatic. It requires a set of true statements $\mathfrak{S}$ for the group $G$ to be provided before attempting a proof. In the next subsection, we will discuss how to establish such statements using a finite model finder.
    \item The proposed approach is limited to establishing the non-orderability of groups. While many groups are orderable, proving it involves second-order reasoning that includes the quantifier ``there exists an order''. Alternatively, it could be handled by inductive reasoning. Both directions of automation appear to be promising areas for further research, but they are beyond the scope of this paper.
    \item An application of the automated reasoning to orderability can be found in \cite{wehrung2021}. In that work, finite model building was used to establish the orderability of finite monoids. However, this approach cannot be applied to show the left-orderability and bi-orderability of infinite groups.
    
\end{enumerate}

\subsection{Finite models as a source of true statements}

As we have already noticed, the proposed methodology is not complete and is not fully automated. The choice of a set of true statements (inequalities) $\mathfrak{S}$ in a group of interest remains a crucial and, generally, creative step. In applications one may use any known equationally expressible property of the group, such as non-commutativity.

We propose here a partial automation of the search for true statements using an automated reasoning technique, \emph{finite model finding}.  
For a $G$ presented by $\langle S|R\rangle$, a model of $\Gr \cup \Ax_R$ can be viewed as a quotient group of $G$.  
Hence any inequality $t_{1} \neq t_{2}$ among ground terms $t_{1}$ and $t_{2}$ which holds true in a model of $\Gr \cup \Ax_R$ also holds true in $G$.  

The proposed approach then works as follows: for a group presentation $\langle S|R\rangle$, we search for finite models of $\Gr \cup \Ax_R$ using an automated finite model finder tool, such as Mace4. If a finite model $G'$  is discovered, we can take any subset of ground inequalities true in $G'$ as a set of true statements $\mathfrak{S}$ in $G$. 

If a group of interest $G$ does not have nontrivial finite quotients, then the finite model finder will never find a useful model. Hence this approach is incomplete. Empirically, it has been effective in many, though not all, of our experiments.

\subsection{Examples}
Now we show how the proposed first principles approach works on some simple examples. 
\begin{example}
The fundamental group of Klein bottle has a presentation \[\langle a, b\; |\; a^{-1}ba = b^{-1}  \rangle.\] This group  is known to be left-orderable, but not bi-orderable \cite{boyer2005orderable}. One can prove that the group is not bi-orderable by noticing that one has to have both $b < e$ and $e < b$, which is impossible. This argument has to be complemented by a proof of the fact that $b\neq e$. If $b=e$ in the group, then $a$ becomes a generator. Thus it suffices to prove that the Klein bottle group is not cyclic.

If we want to apply automated reasoning, the corresponding theory $\Gr\cup \Ax_R\cup\AxL\cup \OrdB$ can be formulated as follows in the syntax of Prover9:
\begin{verbatim}
% Gr                            % Ax_R
(x * y) * z = x * (y * z).      (a' * b) * a = b'.
x * e = x.                        
e * x = x.                        
x' * x = e. 
x * x' = e.


% AxL                            % OrdB                            
- L(x,x).                        L(x,y) -> L((z*x)*u,(z*y)*u).         
L(x,y) & L(y,z) -> L(x,z).  
(x=y) | L(x,y) | L(y,x).
\end{verbatim} 
 One can notice that it is impossible to prove contradiction from such a theory because it is consistent and has a one-element model, namely, the trivial group. In order to get a contradiction, one needs to add some true statements in the group ($\mathfrak{S}$ in Proposition~\ref{prop:main}) to the theory. 
 
 In this example, $b\neq e$ is sufficient. To prove $b\neq e$ in the Klein bottle group automatically, we use the finite model building technique introduced in the previous subsection. When asked if the theory $\Gr \cup \Ax_{R} \cup \{b\neq e\}$ has a model (Task 1.1), the model builder Mace4 produces a model of size $2$; see Table~\ref{tab:summary2}. Therefore, the inequality $b\neq e$ is confirmed in the Klein bottle group. 
 
 We can use Prover9 to prove contradiction (Task 1.2) from the theory $\Gr\cup \Ax_R\cup\AxL\cup \OrdB\cup\{b\neq e\}$. Note that Prover9 rediscovers the human-authored proof given above, finding a contradiction after deriving both $L(b,e)$ and $L(e,b)$.
\end{example}

\begin{example}~\label{ex:sl2}
Special linear group $SL_{2}(\mathbb{Z})$ has a presentation
\[ \langle a,b \;|\; a^{4}=e, (ba)^{3} = b^{2}
\rangle\] and is known to be non-left-orderable because $a$ is a nontrivial torsion.

We prove this fact automatically. First we validate that both $a\neq e$ and $b\neq e$ hold true by finding models (Task 2.1) for the theory $\Gr \cup \Ax_R\cup \{a\neq e, b\neq e\}$ using Mace4. Then by adding $\{a\neq e, b\neq e\}$ to the theory $\Gr \cup \Ax_R\cup \AxL\cup \OrdL$, we derive contradiction (Task 2.2) using Prover9.
\end{example}

\begin{example}\label{ex:Fibonacci}
    The $n$-th Fibonacci group $F(2,n)$ ($n\ge 2$) has a presentation 
    \[\langle a_0,\ldots, a_{n-1} \;|\; a_i a_{i+1}= a_{i+2} \mbox{ for } i=0,\ldots, n-1\rangle.\]
    When $n$ is odd, this group contains a nontrivial torsion by \cite[Proposition 3.1]{bardakov2003generalization}, which implies the non-left-orderability. When $n$ is even, this group is the fundamental group of a cyclic branched cover of the figure-eight knot by \cite[Theorem 1]{hilden1992arithmeticity}, and is not left-orderable by \cite[Theorem 2]{dkabkowski2005non}.

    While our methods cannot automatically verify the non-left-orderability for every positive integer $n\ge 2$, we can apply them to some relatively large integers. These instances could serve as inspiration for mathematicians to establish non-left-orderability in general cases.

    Suppose that $n=12$. Define the set of inequalities $\mathfrak{S}$ by
    \[\mathfrak{S}:=\{a_i\neq e : i=0,\ldots,n-1\}.\]
    Using Mace4, we prove that $\Gr \cup \Ax_R\cup \mathfrak{S}$ is consistent (Task 3.1). Using Prover9, we prove that $\Gr \cup \Ax_R\cup \AxL \cup \OrdL \cup \mathfrak{S}$ is inconsistent (Task 3.2). Thus the Fibonacci group $F(2,12)$ is not left-orderable.

    When $n=11$, the inconsistency of $\Gr \cup \Ax_R\cup \AxL \cup \OrdL \cup \mathfrak{S}$ can still be verified (Task 3.3) using Prover9. However, it is difficult to find a finite model of $\Gr \cup \Ax_R\cup\mathfrak{S}$ using Mace4. To deal with the computational challenge, one approach is to manually deduce these inequalities based on the fact that $F(2,11)$ is infinite \cite{chalk1998fibonacci}.

\end{example}

\begin{example}\label{ex:B3}
    The braid group $B_3$ on three strands is isomorphic to the knot group of the trefoil knot $T_{2,3}$, and have a presentation \[\langle a,b\;|\; aba=bab \rangle.\]
    This group does not admit bi-invariant circular orders, according to \cite[Corollary 8.8]{ba2023knot}. An alternative proof follows from the left-orderability (thus they are torsion-free) and the non-bi-orderability of braid groups, as well as knot groups of nontrivial torus knots. According to \cite[Proposition 3]{zheleva1976cyclically}, a torsion-free group admits a bi-invariant circular order if and only if it is bi-orderable.

    To prove this fact automatically, we define the set of inequalities $\mathfrak{S}$ by
    \[\mathfrak{S}:=\{a\neq e,b\neq e,a \neq b)\}.\]
    Using Mace4, we prove that $\Gr \cup \Ax_R\cup \mathfrak{S}$ is consistent (Task 4.1). Using Prover9, we prove that $\Gr \cup \Ax_R\cup \AxC \cup \OrdCB \cup \mathfrak{S}$ is inconsistent (Task 4.2). Thus the group $B_3$ does not admit bi-invariant circular orders.
\end{example}

\begin{example} \label{ex:D7}
To illustrate how our approach can be applied to circular non-orderability, consider the dihedral group $D_{7}$ of order $14$ given by the presentation 
\[\langle a, b\; |\; a^{7} =e, b^{2} =e, bab=a^{-1} \rangle.\] This group is finite and non-cyclic, thus by \cite[Theorem 1]{zheleva1976cyclically}, it is not circularly orderable. 

We prove this fact directly using automated reasoning. Define the set of inequalities $\mathfrak{S}$ by \[\mathfrak{S}:= \{a\neq e, b\neq e, a\neq b\}.\] Using Mace4, we prove that $\Gr \cup \Ax_R\cup \mathfrak{S}$ is consistent (Task 5.1). Using Prover9, we prove that $\Gr \cup \Ax_R\cup \AxC \cup \OrdC \cup \mathfrak{S}$ is inconsistent (Task 5.2). Thus the group $D_{7}$ is not circularly orderable.

Note that this example already takes much longer to compute than the previous ones; see Table~\ref{tab:summary1}. To improve the applicability of our approach and handle more complicated examples, we present some methods in the subsequent section to deal with the computational challenge. 
\end{example} 

\section{Weakened theories and positive cones}
\subsection{Weakened theories}\label{subsec:weakened}
Establishing the non-orderability using automated theorem provers is inherently incomplete methodology. It may fail for various reasons:
\begin{enumerate}[label=(\alph*)]
    \item If the orderability status of a group is unknown, it may turned out to be orderable. In the case it is not possible to derive contradiction.
    \item The group may turned out to be non-orderable, but a supplied set of true statements is not sufficient to derive contradiction.
    \item The supplied set of true statements may be sufficient to derive a contradiction, but it takes too long to find a proof automatically.
\end{enumerate}
In practice, it is difficult to distinguish between these alternatives. Hence to improve utility of the methodology one needs to consider possible optimisations to improve efficiency of the proof search. 

One of possible optimisation is based on using weaker theory to derive contradictions. For example, in a typical derivation of contradiction in the`first principle approach the axiom
\begin{enumerate} [label=(\alph*)]
\item[(c)] $\forall x \forall y ((x=y) \lor (x < y)\lor(y<x))$ \hfill\emph{(connectedness)} 
\end{enumerate}
in $\AxL$ can be used with a supplied inequality  $t_{1} \neq t_{2}$ to derive $(t_{1} < t_{2}) \lor (t_{2} < t_{1}).$ 
So, the search space for proof can potentially be reduced by removing this axiom and replacing supplied inequality $t_{1} \neq t_{2}$ with $(t_{1} < t_{2}) \lor (t_{2} < t_{1}).$ In summary, we have the following proposition.

\begin{proposition} \label{prop:weak}
    Let $G$ be a group with presentation $\langle S|R\rangle$. Let $\mathfrak{S}$ be a set of true statements for $G$. Let $\mathfrak{S}^{<}$ be a set of formulas obtained by replacing all inequalities $t_{1} \neq t_{2}$ from $\mathfrak{S}$ by the corresponding formulas  $(t_{1} < t_{2}) \lor (t_{2} < t_{1})$. Let $\AxL'$ denote $\AxL$ minus the axiom of connectedness. Then the following statements hold:
    \begin{enumerate}[label=(\alph*)]
        \item If $\Gr \cup \Ax_R\cup \AxL'\cup \OrdL \cup \mathfrak{S}^<$ is inconsistent, then $G$ is not left-orderable.
        \item If $\Gr \cup \Ax_R\cup\AxL' \cup \OrdB \cup \mathfrak{S}^<$ is inconsistent, then $G$ is not bi-orderable.
    \end{enumerate}
\end{proposition}

One can derive similar results for circular orders and bi-invariant circular orders.

\begin{proposition} \label{weak-cbo}
Let $G$ be a group with presentation $\langle S|R\rangle$. Let $\mathfrak{S}$ be a set of true statements for $G$. Let $\mathfrak{S}^{C}$ be a set of formulas obtained by replacing each triple of inequalities $\{ t_{1} \neq t_{2}, t_{2} \neq t_{3}), t_{3} \neq t_{1}\}$ from $\mathfrak{S}$ by the corresponding formula  $C(t_1, t_2,t_3) \lor C(t_1,t_3,t_2)$. Let $\AxC'$ denote $\AxC$ minus the axiom of connectedness. Then the following statements hold:
    \begin{enumerate}[label=(\alph*)]
        \item If $\Gr \cup \Ax_R\cup \AxC'\cup \OrdC \cup \mathfrak{S}^C$ is inconsistent, then $G$ is not circularly orderable.
        \item If $\Gr \cup \Ax_R\cup\AxC' \cup \OrdCB \cup \mathfrak{S}^C$ is inconsistent, then $G$ does not admit a bi-invariant circular order.
    \end{enumerate}
\end{proposition} 

To improve the efficiency further, we can strengthen the assumption $(t_{1} < t_{2}) \lor (t_{2} < t_{1})$ in $\mathfrak{S}^<$ to $t_{1} < t_{2}$, or strengthen the assumption $C(t_1, t_2,t_3) \lor C(t_1,t_3,t_2)$ in $\mathfrak{S}^C$ to $C(t_1, t_2,t_3)$. By imposing stronger assumptions, we are testing the existence of an order satisfying extra formulas. A contradiction of a strengthened theory leads to a partial result on non-orderability. However, the first strengthening is free by symmetry.

\begin{proposition} \label{symmetry}
If we strengthen one formula of form $(t_{1} < t_{2}) \lor (t_{2} < t_{1})$ in $\mathfrak{S}^<$ to $t_{1} < t_{2}$, or of form $C(t_1, t_2,t_3) \lor C(t_1,t_3,t_2)$ in $\mathfrak{S}^C$ to $C(t_1, t_2,t_3)$, then the conclusions in Proposition~\ref{prop:weak} and Proposition~\ref{weak-cbo} still hold.
\end{proposition}
\begin{proof}
    For each linear order $<$, we say $x<_{op}y$ if and only if $y<x$. Then all axioms in $\AxL$, $\OrdL$, and $\OrdB$ hold invariant when replacing $<$ with $<_{op}$. Thus if $<$ is a left order (resp., a bi-order), then $<_{op}$ is also a left order (resp., a bi-order). Either $<$ or $<_{op}$ satisfies the formula $t_1<t_2$.

    For a cyclic order $C$, we say $C_{op}(x,y,z)$ if and only if $C(x,z,y)$. Then all axioms in $\AxC$, $\OrdC$, and $\OrdCB$ hold invariant when replacing $C$ with $C_{op}$. Thus if $<$ is a circular order (resp., a bi-invariant circular order), then $<_{op}$ is also a circular order (resp., a bi-invariant circular order). Either $C$ or $C_{op}$ satisfies the formula $C(t_1, t_2,t_3)$.
\end{proof}

\subsection{Positive cones} \label{subsec:positive_cone}
In this subsection, we incorporate a well-known technique in the theory of ordered groups into our methodology: the positive cone technique. A positive cone of an order $<$ on the group $G$ is defined as the set of all positive elements. In other words, it is the set $\{x\in G: e < x\}$. This concept is particularly useful for left orders and bi-orders, as the positive cone determines the left order or the bi-order. According to the left-invariance, we can show that $x<y$ if and only if $x^{-1}y$ is in the positive cone. By translating the axioms for orders to the corresponding axioms for positive cones, we can gain a computational advantage, as it reduces a binary predicate to a unary one.

The positive cone of a left order or a bi-order satisfies the following axioms, which we denote by $\AxPC$:
\begin{enumerate} [label=(\alph*)]
\item $\neg P(e)$,  \hfill\emph{(irreflexivity)} 
\item $\forall x \forall y (P(x) \land  P(y) \to  P(x\cdot y)$), \hfill\emph{(closure)} 
\item $\forall x ((x=e) \lor P(x)\lor P(x'))$. \hfill\emph{(connectedness)} 
\end{enumerate}

Additionally, the positive cone of a bi-order satisfies the conjugacy invariance axiom $\PB$:
\begin{enumerate} [label=(\alph*)]
\item[(d)] $\forall x \forall y (P(x) \to P((y\cdot x)\cdot y' ) )$. \hfill\emph{(conjugacy invariance)} 
\end{enumerate}

We can generalise the positive cone method to circular orders. We define the positive cone of a circular order $C$ on the group $G$ as the set $\{(x,y)\in G: C(e,x,y)\}$. Because $C(x,y,z)$ if and only if $(x^{-1}y,x^{-1}z)$ is in the positive cone, the positive cone determines the circular order. Thus we can translate the axioms for circular orders and bi-invariant circular orders to the axioms for their positive cones as follows.

The positive cone of a circular order satisfies the following axioms, which we denote by $\AxPCL$:
\begin{enumerate} [label=(\alph*)]
\item $\forall x \forall y (P(x,y) \to  P(x'\cdot y,x'))$, \hfill\emph{(cyclicity)} 
\item $\forall x (\neg P(x, x))$,  \hfill\emph{(irreflexivity)} 
\item $\forall x \forall y\forall z (P(x,y)\land P(y,z) \to P(x,z))$,  \hfill\emph{(transitivity)} 
\item $\forall x \forall y ((e=x)\lor(e=y)\lor(x=y)\lor P(x,y)\lor P(y,x)) $. \hfill\emph{(connectedness)} 
\end{enumerate}

Additionally, the positive cone of a bi-invariant circular order satisfies the conjugacy invariance axiom $\PCB$:
\begin{enumerate} [label=(\alph*)]
\item[(e)] $\forall x \forall y\forall z (P(x,y) \to P((z \cdot x)\cdot z' ),(z \cdot y)\cdot z' ) $. \hfill\emph{(conjugacy invariance)} 
\end{enumerate}

We encourage the readers to verify the following lemmas.

\begin{lemma}\label{lem:1}
    Assume the axioms in $\Gr$ holds for binary function $\cdot$, unary function $'$ and constant $e$. Let $<$ be a binary predicate satisfying $\OrdL$ or $\OrdB$. Define a unary predicate $P$ by $P(x)$ if and only if $e<x$. Then
    \begin{enumerate}[label=(\alph*)]
        \item the irreflexivity axiom in $\AxL$ implies the irreflexivity axiom in $\AxPC$;
        \item the transitivity axiom in $\AxL$ implies the closure axiom in $\AxPC$;
        \item the connectedness axiom in $\AxL$ implies the connectedness axiom in $\AxPC$;
        \item the axiom $\OrdB$ implies the axiom $\PB$.
    \end{enumerate} 
\end{lemma}

\begin{lemma}\label{lem:2}
    Assume the axioms in $\Gr$ holds for binary function $\cdot$, unary function $'$ and constant $e$. Let $P$ be a unary predicate. Define a binary predicate $<$ by $x<y$ if and only if $P(x' \cdot y)$. Then 
    \begin{enumerate}[label=(\alph*)]
        \item the axiom $\OrdL$ holds;
        \item the irreflexivity axiom in $\AxPC$ implies the irreflexivity axiom in $\AxL$;
        \item the closure axiom in $\AxPC$ implies the transitivity axiom in $\AxL$;
        \item the connectedness axiom in $\AxPC$ implies the connectedness axiom in $\AxL$;
        \item the axiom $\PB$ implies the axiom $\OrdB$.
    \end{enumerate}
\end{lemma}

\begin{lemma}\label{lem:3}
     Assume the axioms in $\Gr$ holds for binary function $\cdot$, unary function $'$ and constant $e$. Let $C$ be a ternary predicate satisfying $\OrdC$ or $\OrdCB$. Define a binary predicate $P$ by $P(x,y)$ if and only if $C(e,x,y)$. Then
    \begin{enumerate}[label=(\alph*)]
        \item the cyclicity axiom in $\AxC$ implies the cyclicity axiom in $\AxPCL$;
        \item the irreflexivity axiom in $\AxC$ implies the irreflexivity axiom in $\AxPCL$;
        \item the transitivity axiom in $\AxC$ implies the transitivity axiom in $\AxPCL$;
        \item the connectedness axiom in $\AxC$ implies the connectedness axiom in $\AxPCL$;
        \item the axiom $\OrdCB$ implies the axiom $\PCB$.
    \end{enumerate} 
\end{lemma}

\begin{lemma}\label{lem:4}
     Assume the axioms in $\Gr$ holds for binary function $\cdot$, unary function $'$ and constant $e$. Let $P$ be a binary predicate. Define a ternary predicate $C$ by $C(x,y,z)$ if and only if $P(x'\cdot y,x'\cdot z)$. Then
    \begin{enumerate}[label=(\alph*)]
        \item the axiom $\OrdC$ holds;
        \item the cyclicity axiom in $\AxPCL$ implies the cyclicity axiom in $\AxC$;
        \item the irreflexivity axiom in $\AxPCL$ implies the irreflexivity axiom in $\AxC$;
        \item the transitivity axiom in $\AxPCL$ implies the transitivity axiom in $\AxC$;
        \item the connectedness axiom in $\AxPCL$ implies the connectedness axiom in $\AxC$;
        \item the axiom $\PCB$ implies the axiom $\OrdCB$.
    \end{enumerate} 
\end{lemma}

By replacing $\AxL\cup\OrdL$ with $\AxPC$, replacing $\AxL\cup\OrdB$ with $\AxPC\cup \PB$, replacing $\AxC\cup \OrdC$ with $\AxPCL$, and replacing $\AxC\cup \OrdCB$ with $\AxPCL\cup \PCB$, we can translate Proposition~\ref{prop:main}, Proposition~\ref{prop:weak}, Proposition~\ref{weak-cbo}, and Proposition~\ref{symmetry} into positive cone forms. We prove that these positive cone translations lead to \emph{equiconsistent} theories.

\begin{proposition}\label{prop:positive_cone}
    Let $s_i$, $t_i$ $(i=0,1,\ldots,k)$ be ground terms. Let $\mathfrak{S}$ denote the set of inequalities \[\{s_i\neq t_i: i=0,1,\ldots,k\}.\] Let $\mathfrak{S}^<$ denote the set of axioms \[\{s_0< t_0\}\cup\{(s_i<t_i)\lor (t_i<s_i): i=1,\ldots, k\}.\] Let $\mathfrak{S}^P$ denote the set of axioms \[\{P(s_0' \cdot t_0)\}\cup\{P(s_i'\cdot t_i)\lor P(t_i'\cdot s_i): i=1,\ldots, k\}.\]
    Let $\AxPC'$ (resp. $\AxL'$) denote $\AxPC$ (resp. $\AxL$) minus the axiom of connectedness. Then     
    \begin{enumerate}[label=(\alph*)]
        \item $\Gr \cup \Ax_R\cup \AxL\cup \OrdL \cup \mathfrak{S}$ is consistent if and only if $\Gr\cup \Ax_R\cup \AxPC\cup \mathfrak{S}$ is consistent;
        \item $\Gr \cup \Ax_R\cup\AxL \cup \OrdB \cup \mathfrak{S}$ is consistent if and only if $\Gr\cup \Ax_R\cup \AxPC\cup \PB\cup \mathfrak{S}$ is consistent;
        \item $\Gr \cup \Ax_R\cup \AxL'\cup \OrdL \cup \mathfrak{S}^<$ is consistent if and only if $\Gr\cup \Ax_R\cup \AxPC'\cup \mathfrak{S}^P$ is consistent;
        \item $\Gr \cup \Ax_R\cup\AxL' \cup \OrdB \cup \mathfrak{S}^<$ is consistent if and only if $\Gr\cup \Ax_R\cup \AxPC'\cup \PB\cup \mathfrak{S}^P$ is consistent.
    \end{enumerate}
\end{proposition}
\begin{proof}
    If a theory $T_1$ involving $<$ (in either of the four cases) is consistent, then there is a model $\mathfrak{M}_1$ of it. We extend $\mathfrak{M}_1$ by the interpretation of $P$, defined as $P(x)$ if and only if $e<x$, then remove the interpretation of $<$. Then by Lemma~\ref{lem:1}, the resulting model $\mathfrak{M}_2$ is a model of the corresponding theory $T_2$ involving $P$.

    Conversely, if a theory $T_2$ involving $P$ is consistent, then there is a model $\mathfrak{M}_2$ of it. We interpret $<$ by $x<y$ if and only if $P(x'\cdot y)$ and then remove the interpretation of $P$. Then by Lemma~\ref{lem:2}, the resulting model $\mathfrak{M}_1$ is a model of the corresponding theory $T_1$ involving $<$.
\end{proof}

Similarly, by Lemma~\ref{lem:3} and Lemma~\ref{lem:4}, we have the following proposition.

\begin{proposition}\label{prop:positive_cone_C}
    Let $r_i$, $s_i$, $t_i$ $(i=1,\ldots,k)$ be ground terms. Let $\mathfrak{S}$ denote the set of inequalities \[\{r_i\neq s_i,s_i\neq t_i, t_i\neq r_i:i=1,\ldots,k\}.\] Let $\mathfrak{S}^C$ denote the set of axioms \[\{C(r_0,s_0,t_0\}\cup\{C(r_i,s_i,t_i)\lor C(r_i,s_i,t_i): i=1,\ldots, k\}.\] Let $\mathfrak{S}^P$ denote the set of axioms \[\{P(r_0'\cdot s_0,r_0' \cdot t_0)\}\cup\{P(r_i'\cdot s_i,r_i' \cdot t_i)\lor P(r_i' \cdot t_i,r_i'\cdot s_i): i=1,\ldots, k\}.\] Let $\AxPCL'$ (resp. $\AxC'$) denote $\AxPCL$ (resp. $\AxC$) minus the axiom of connectedness. 
    Then     
    \begin{enumerate}[label=(\alph*)]
        \item $\Gr \cup \Ax_R\cup \AxC \cup \OrdC \cup \mathfrak{S}$ is consistent if and only if $\Gr\cup \Ax_R\cup \AxPCL\cup \mathfrak{S}$ is consistent;
        \item $\Gr \cup \Ax_R\cup\AxC \cup \OrdCB \cup \mathfrak{S}$ is consistent if and only if $\Gr\cup \Ax_R\cup \AxPCL\cup \PCB\cup \mathfrak{S}$ is consistent;
        \item $\Gr \cup \Ax_R\cup \AxC' \cup \OrdC \cup \mathfrak{S}^C$ is consistent if and only if $\Gr\cup \Ax_R\cup \AxPCL'\cup \mathfrak{S}^P$ is consistent;
        \item $\Gr \cup \Ax_R\cup\AxC' \cup \OrdCB \cup \mathfrak{S}^C$ is consistent if and only if $\Gr\cup \Ax_R\cup \AxPCL'\cup \PCB\cup \mathfrak{S}^P$ is consistent.
    \end{enumerate}
\end{proposition}

Proposition~\ref{prop:positive_cone} and Proposition~\ref{prop:positive_cone_C} offer an alternative approach to demonstrating the inconsistency of the theories in  Proposition~\ref{prop:main} and Proposition~\ref{symmetry} through automated reasoning, which usually enhances efficiency.

It is worth noting that the performance of the automated theorem prover Prover9 on a single task is highly unpredictable. For further optimisations, one may fine tune the proof search strategy or attempt alternative presentations of the same group.

\subsection{Examples}
\begin{example}
    In Example~\ref{ex:D7}, we established the non-circular-orderability of the dihedral group $D_{7}$ using automated reasoning. However, the process of deducing the inconsistency of $\Gr\cup \Ax_R\cup \AxC\cup\OrdC \cup \mathfrak{S}$ with Prover9 is time-consuming. By the statement (a) in Proposition~\ref{prop:positive_cone_C}, we can alternatively demonstrate the inconsistency of $\Gr \cup \Ax_R \cup \AxPCL \cup \mathfrak{S}$ (Task 6.1). The positive cone translation significantly reduces the time required; see Table~\ref{tab:summary1}.
    
\end{example}

\begin{example}\label{ex:sl2-weak}
    Consider the special linear group $SL_2(\mathbb{Z})$ with the presentation as shown in Example~\ref{ex:sl2}. By \cite[Theorem 5.10]{giraudet2018first}, in a group with a bi-invariant circular order, the torsion part is central. Since $a$ is a torsion element in $SL_2(\mathbb{Z})$ with $ab\neq ba$, it follows that $SL_2(\mathbb{Z})$ does not admit a bi-invariant circular order. 

    To prove this fact automatically, we first establish the inequalities $e\neq ab$, $e\neq ba$, $ab\neq ba$ by building a model (Task 7.1) of $\Gr\cup \Ax_R\cup \{e\neq a\cdot b,e\neq b\cdot a,a\cdot b\neq b\cdot a\}$ by Mace4. Then we prove that $\Gr\cup \Ax_R\cup \AxPCL'\cup \PCB\cup\{P(a\cdot b,b\cdot a)\}$ is inconsistent (Task 7.2) by Prover9. Thus by Proposition~\ref{weak-cbo}, Proposition~\ref{symmetry} and the statement (d) in Proposition~\ref{prop:positive_cone_C}, the special linear group $SL_2(\mathbb{Z})$ does not admit a bi-invariant circular order. 
\end{example}

\begin{example} \label{ex:Poincare}
The fundamental group of the Poincar\'e homology sphere has a presentation 
\[\langle a,b\;|\; (ab)^{2} = a^{3}, a^{3} = b^{5} \rangle.\]
This group is finite (of order $120$), hence it is not left-orderable.

To prove this fact automatically, we check that $a\neq e$ by finding a model (Task 8.1) of $\Gr\cup \Ax_R\cup \{e\neq a\}$ using Mace4. Then we prove that $\Gr\cup\Ax_R\cup \AxPC'\cup\{P(a)\}$ is inconsistent (Task 8.2) by Prover9. Thus by Proposition~\ref{prop:weak}, Proposition~\ref{symmetry} and the statement (a) in Proposition~\ref{prop:positive_cone}, the fundamental group of the Poincar\'e homology sphere is not left-orderable. 
\end{example} 

\begin{example} \label{ex:5_2}
The knot group of the knot $5_{2}$ has a presentation \[\langle a,b \;|\; b^{2}a^{2}b^{2} = a b^{3}a\rangle.\] This group is known to be left-orderable, as all knot groups, and not bi-orderable, according to \cite[page 5]{chiswell2015residual} and \cite[Theorem 7]{naylor2016generalized}. For the same reason as in Example~\ref{ex:B3}, this group does not admit a bi-invariant circular order.

We first prove the non-bi-orderability using automated reasoning. A model for $\Gr \cup \Ax_{R} \cup \{b\cdot a \neq (a\cdot b)\cdot b\}$ can be found (Task 9.1) using Mace4. Thus $ba \neq ab^{2}$ holds true in this group. We can verify that $\Gr \cup \Ax_R\cup\AxPC' \cup \PB \cup \{P((b\cdot a)'\cdot ((a\cdot b)\cdot b))\}$ is inconsistent (Task 9.2) by Prover9. By Proposition~\ref{prop:weak}, Proposition~\ref{symmetry} and the statement (d) in Proposition~\ref{prop:positive_cone}, the knot group of $5_{2}$ is not bi-orderable. Alternatively, we may use the inequality $ab \neq ba$ in the automated proof. (Task 9.3 and Task 9.4) 

Now we show how to prove the non-existence of a bi-invariant circular order using automated reasoning. First, we prove $e\neq ab$, $e\neq ba$ and $ab\neq ba$ by building a model (Task 9.5) of $\Gr\cup \Ax_R\cup\{e\neq a\cdot b, e\neq b\cdot a, a\cdot b\neq b\cdot a\}$ using Mace4. Then we prove the inconsistency of $\Gr\cup \Ax_R\cup\AxPCL'\cup \PCB\cup\{P(a\cdot b,b\cdot a)\}$ (Task 9.6) using Prover9. By Proposition~\ref{weak-cbo}, Proposition~\ref{symmetry} and the statement (d) in Proposition~\ref{prop:positive_cone_C}, the knot group of $5_{2}$ does not admit a bi-invariant circular order. Alternatively, we may use the inequalities $e\neq ba$, $e\neq ab^2$ and $ba\neq ab^2$ in the automated proof. (Task 9.7 and Task 9.8)
\end{example} 

\begin{example}\label{ex:weeks}
The fundamental group of the Weeks manifold has a presentation
\[\langle a,b\;|\; a^{2}b^{2}a^{2} = ba^{-1}b,b^2a^{2}b^{2} = ab^{-1}a
\rangle.\]
According to \cite[Theorem 9.2]{calegari2003laminations}, this group is not circularly orderable. While our approach is not sophisticated enough to provide an automated proof of the non-circular-orderability in reasonable time, we can prove two weaker properties using automated reasoning: the non-left-orderability and the absence of bi-invariant circular orders.

By building a finite model (Task 10.1) of $\Gr \cup \Ax_R \cup \{ a\neq e\}$ using Mace4, we obtain that $a\neq e$. By deriving a contradiction (Task 10.2) of $\Gr\cup\Ax_R\cup\AxPC \cup\{a\neq e\}$ using Prover9, it follows from Proposition~\ref{prop:main} and the statement (a) in Proposition~\ref{prop:positive_cone} that, the Weeks manifold group is not left-orderable.

In order to prove the non-existence of bi-invariant circular orders automatically, we need to confirm three inequalities $e\neq ab$, $e\neq ba$ and $ab\neq ba$. The first two inequalities can be verified by building a finite model (Task 10.3) of $\Gr\cup\Ax_R\cup\{e\neq a\cdot b, e\neq b\cdot a\}$ using Mace4. However, the last inequality $ab\neq ba$ turns out to be more challenging. One way to prove $ab\neq ba$ in the Weeks manifold group is through an $SL_2(\mathbb{C})$-representation as described on \cite[page 24]{chinburg2001arithmetic}. Alternatively, we can argue that the equality $ab=ba$ implies that the group of interest is isomorphic to the product of two cyclic groups of order $5$, hence it is finite and non-cyclic. According to \cite[Theorem 1]{zheleva1976cyclically}, in such a case, it is not circularly orderable.

We can prove that the theory $\Gr\cup\Ax_R\cup \AxPCL'\cup \PCB\cup\{P(a\cdot b, b\cdot a)\}$ is inconsistent (Task 10.4) using Prover9. Thus by Proposition~\ref{weak-cbo}, Proposition~\ref{symmetry} and the statement (d) in Proposition~\ref{prop:positive_cone_C}, the Weeks manifold group does not admit a bi-invariant circular order.
\end{example}

\section{Torsions, generalised torsions, and more} \label{sec:torsions}
\subsection{Torsions and generalised torsions}
In this subsection, we show that, if $\mathfrak{S}$ contains a single inequality $t_1\neq t_2$, then the methods of establishing the non-left-orderability (resp. non-bi-orderability) via the weakened theory described in Subsection~\ref{subsec:weakened} is essentially detecting whether $t_1^{-1} t_2$ represents a torsion (resp. generalised torsion) in the group presented by $\langle S|R\rangle$. Note that a nontrivial torsion (resp. generalised torsion) is a well-known obstruction to left-orderability (resp. bi-orderability); see \cite[Proposition 1.3 and Problem 1.22]{clay2016ordered} for example.

First, we present the definition of torsion and establish the desired equivalence.

\begin{definition} 
A group element $x \in G$ is called a torsion if there exists a positive integer $n$ such that $x^{n} =e$ where $e$ is the identity element of the group $G$.
\end{definition}

\begin{proposition}\label{prop-torsion}
    Let $G$ be a group with presentation $\langle S|R\rangle$. Let $t$ be a ground term representing a group element in $G$. Then $\Gr\cup\Ax_R\cup\AxPC'\cup\{P(t)\}$ is inconsistent if and only if $t$ represents a torsion in $G$.
\end{proposition}
\begin{proof}
    We first prove the ``if'' part. If $t$ represents a torsion in $G$, then there exists a positive integer $n$ such that $t^n=e$ holds assuming the axioms $\Gr$ and $\Ax_R$, where the ground term $t^n$ is recursively defined by $t^1=t$ and $t^n=t^{n-1}\cdot t$ for $n\ge 2$. The closure axiom in $\AxPC'$ and the axiom $P(t)$ implies that $P(t^n)$ for every positive integer $n$ inductively, hence we have $P(e)$, which contradicts to the irreflexivity axiom in $\AxPC'$. Therefore $\Gr\cup\Ax_R\cup\AxPC'\cup\{P(t)\}$ is inconsistent if $t$ represents a torsion in $G$.

    Next, we prove the ``only if'' part. Suppose that $t$ represents a non-torsion element $\bar{t}$ in $G$. For any $x\in G$, let $P(x)$ be the proposition that $\bar{t}^n=x$ in $G$ for some positive integer $n$. We prove that the group $G$ together with the predicate $P$ constitutes a model for $\Gr\cup\Ax_R\cup\AxPC'\cup\{P(t)\}$:
    \begin{enumerate}[label=(\alph*)]
        \item The axioms in $\Gr\cup Ax_R$ are satisfied because $G$ is the group with presentation $\langle S|R\rangle$. 
        \item If $P(e)$ holds, then there exists a positive integer $n$ such that $\bar{t}^n=e$, which contradicts to the assumption that $\bar{t}$ is not a torsion. Thus the irreflexivity axiom in $\AxPC'$ is satisfied. 
        \item If $P(x)$ and $P(y)$ holds for $x,y\in G$, then there exist positive integers $m$ and $n$ such that $\bar{t}^m=x$ and $\bar{t}^{n}=y$, so we have $\bar{t}^{m+n}=x$ and therefore $P(xy)$ holds. Thus the transitivity axiom in $\AxPC'$ is satisfied.
        \item The axiom $P(t)$ is satisfied by the definition of $P$.
    \end{enumerate}
    Therefore the theory $\Gr\cup\Ax_R\cup\AxPC'\cup\{P(t)\}$ is consistent if $t$ represents a non-torsion element in $G$.
\end{proof}

Then, we present the definition of generalised torsion and prove the analogous statement to Proposition~\ref{prop-torsion}. For convenience, we consider the identity element $e$ as a generalised torsion.

\begin{definition}  
A group element $x \in G$ is called a generalised torsion if there exist $y_{1}, \ldots, y_{n} \in G$ such that \[(y_{1}xy^{-1}_{1})(y_{2}xy^{-1}_{2})\cdots (y_{n}xy^{-1}_{n})=e,\] where $e$ is the identity element of of the group $G$.
\end{definition}  

\begin{proposition}\label{prop-generalised-torsion}
    Let $G$ be a group with presentation $\langle S|R\rangle$. Let $t$ be a ground term representing a group element in $G$. Then $\Gr\cup\Ax_R\cup\AxPC'\cup \PB\cup\{P(t)\}$ is inconsistent if and only if $t$ represents a generalised torsion in $G$.
\end{proposition}
\begin{proof}
    We first prove the ``if'' part. If $t$ represents a generalised torsion $\bar{t}$ in $G$, then there exist $y_{1}, \ldots, y_{n} \in G$ such that \[(y_{1}\bar{t} y^{-1}_{1})(y_{2}\bar{t} y^{-1}_{2})\cdots (y_{n}\bar{t} y^{-1}_{n})=e.\] By the definition of $\langle S|R \rangle$, there exist ground terms $t_1, \ldots, t_n$, such that the product of $(t_i\cdot \bar{t})\cdot t_i'$ $(i=1,\ldots, n)$ equals to $e$, assuming the axioms $\Gr$ and $\Ax_R$. The conjugacy invariance axiom $\PB$ and $P(t)$ implies that $P((t_i\cdot \bar{t})\cdot t_i')$ for each $i=1,\ldots, n$. The closure axiom in $\AxPC'$ implies $P(e)$ by induction, which contradicts to the irreflexivity axiom in $\AxPC'$. Therefore $\Gr\cup\Ax_R\cup\AxPC'\cup \PB \cup\{P(t)\}$ is inconsistent if $t$ represents a generalised torsion in $G$.

    Next, we prove the ``only if'' part. Suppose that $t$ represents a element $\bar{t}$ in $G$ that is not a generalised torsion. For any $x\in G$, let $P(x)$ be the proposition that \[(y_1\bar{t}y_1^{-1})(y_2 \bar{t} y_2^{-1})\cdots(y_n \bar{t} y_n^{-1})=x\] in $G$ for some $y_1,\ldots, y_n\in G$. We prove that the group $G$ together with the predicate $P$ constitutes a model for $\Gr\cup\Ax_R\cup\AxPC'\cup \PB \cup\{P(t)\}$:
    \begin{enumerate}[label=(\alph*)]
        \item The axioms in $\Gr\cup Ax_R$ are satisfied because $G$ is the group with presentation $\langle S|R\rangle$. 
        \item Since $\bar{t}$ is not a generalised torsion, we have $\neg P(e)$. Thus the irreflexivity axiom in $\AxPC'$ is satisfied. 
        \item If $P(x)$ and $P(y)$ holds for $x,y\in G$, then there exist $y_1,\ldots, y_m\in G$ and $z_1,\ldots, z_n\in G$ such that
        \[(y_1\bar{t}y_1^{-1})(y_2 \bar{t} y_2^{-1})\cdots(y_m \bar{t} y_m^{-1})=x\]
        and
        \[(z_1\bar{t}z_1^{-1})(z_2 \bar{t} z_2^{-1})\cdots(z_m \bar{t} z_n^{-1})=y,\]
        so we have  \[(y_1\bar{t}y_1^{-1})(y_2 \bar{t} y_2^{-1})\cdots(y_m \bar{t} y_m^{-1})(z_1\bar{t}z_1^{-1})(z_2 \bar{t} z_2^{-1})\cdots(z_m \bar{t} z_n^{-1})=xy\]
        and therefore $P(xy)$ holds. Thus the transitivity axiom in $\AxPC'$ is satisfied.
        \item If $P(x)$ holds for $x\in G$, then there exist $y_1,\ldots, y_n\in G$ such that
        \[(y_1\bar{t}y_1^{-1})(y_2 \bar{t} y_2^{-1})\cdots(y_n \bar{t} y_n^{-1})=x.\] Then for any $y\in G$ we have \[((y y_1)\bar{t}(y y_1)^{-1})((y y_2) \bar{t} (y y_2)^{-1})\cdots((y y_n) \bar{t} (y y_n)^{-1})=y x y^{-1}.\] Thus the conjugacy invariant axiom $\PB$ is satisfied.
        \item The axiom $P(t)$ is satisfied by the definition of $P$.
    \end{enumerate}
    Therefore the theory $\Gr\cup\Ax_R\cup\AxPC'\cup\PB\cup\{P(t)\}$ is consistent if $t$ does not represent a generalised torsion in $G$.
\end{proof}
\subsection{Analogous statement for bi-invariant circular orders}
In this subsection, we establish an analogous statement for bi-invariant circular orders. We give an equivalent condition for the inconsistency of a theory where the axiom of connectedness is weakened as described in Subsection~\ref{subsec:weakened} and the axiom of cyclicity is removed. 

\begin{proposition}~\label{prop-CBO-monoid}
    Let $G$ be a group with presentation $\langle S|R\rangle$. Let $t_1$ and $t_2$ be ground terms representing the group elements $\bar{t}_1$ and $\bar{t}_2$ in $G$ respectively. Let $\overline{\AxPCL}'$ denote $\AxPCL$ minus the axioms of cyclicity and connectedness. Then $\Gr\cup\Ax_R\cup\overline{\AxPCL}'\cup \PCB\cup\{P((t_2 \cdot t_1)\cdot t_2',t_1)\}$ is inconsistent if and only if $\bar{t}_2^{-1}$ is in the monoid generated by $\bar{t}_2$ and the centraliser of $\bar{t}_1$.
\end{proposition}
\begin{proof}
    We first prove the ``if'' part. If $\bar{t}_2^{-1}$ is in the monoid generated by $\bar{t}_1$, $\bar{t}_1^{-1}$, and $\bar{t}_2$, then there exist a nonnegative integer $n$ and some ground terms $s_0=t_2, s_1,s_2,\ldots s_n$, such that $s_i$ ($i=1,2,\ldots,n$) is either $s_{i-1}\cdot t_2$ or $s_{i-1}\cdot t$ where $t$ and $t_1$ represent commutative elements in $G$, and that $s_n = e$ assuming $\Gr$ and $\Ax_R$. We prove $P((s_i\cdot t_1)\cdot s_i', t_1)$ ($i=0,1,\ldots, n$) inductively.

    For $i=0$, the statement $P((s_0\cdot t_1)\cdot s_0', t_1)$ holds true by assumption. Suppose that the statement $P((s_{i-1}\cdot t_1)\cdot s_{i-1}', t_1)$ holds true for some $i\in\{0,1,\ldots,n-1\}$. If $s_i= s_{i-1}\cdot t_2$, then we have \[(s_{i}\cdot t_1)\cdot s_{i}'=(s_{i-1}\cdot ((t_2\cdot t_1)\cdot t_2'))\cdot s_{i-1}'\]
    assuming $\Gr$. By the conjugacy invariance axiom $\PCB$ and $P((t_2\cdot t_1)\cdot t_2', t_1)$, we have $P((s_{i}\cdot t_1)\cdot s_{i}',  ((s_{i-1}\cdot t_1)\cdot s_{i-1}'))$. By the transitivity axiom in $\overline{\AxPCL}'$ and the inductive hypothesis, the statement $P((s_{i}\cdot t_1)\cdot s_{i}', t_1)$ holds true in this case. If $s_i= s_{i-1}\cdot t$ where $t$ and $t_1$ represent commutative elements in $G$, then we have \[(s_{i}\cdot t_1)\cdot s_{i}'=(s_{i-1}\cdot t_1)\cdot s_{i-1}'\] assuming $\Gr$, thus $P((s_{i}\cdot t_1)\cdot s_{i}', t_1)$ also holds true. 

    By taking $i=n$, we have $P(t_1,t_1)$, which contradicts to the irreflexivity axiom in $\overline{\AxPCL}'$. Therefore $\Gr\cup\Ax_R\cup\overline{\AxPCL}'\cup \PCB\cup\{P((t_2 \cdot t_1)\cdot t_2',t_1)\}$ is inconsistent if $\bar{t}_2^{-1}$ is in the monoid generated by $\bar{t}_2$ and the centraliser of $\bar{t}_1$.

    Next, we prove the ``only if'' part. Suppose that $\bar{t}_2^{-1}$ is not in the monoid generated by $\bar{t}_2$ and the centraliser of $\bar{t}_1$. For any $x, y\in G$, let $P(x, y)$ be the proposition that there exist $z_1, z_2\in G$, such that:
    \begin{enumerate}[label=(\alph*)]
        \item $x= z_1 z_2 \bar{t}_2 \bar{t}_1 \bar{t}_2^{-1} z_2^{-1} z_1^{-1}$, 
        \item $y=z_1 \bar{t}_1 z_1^{-1}$, and
        \item $z_2$ is in the monoid generated by $\bar{t}_2$ and the centraliser of $\bar{t}_1$.
    \end{enumerate}
    We prove that the group $G$ together with the predicate $P$ constitutes a model for $\Gr\cup \Ax_R \cup \overline{\AxPCL}' \cup \PCB \cap\{P((t_2\cdot t_1) \cdot t'_2,t_1)\}$:    \begin{enumerate}[label=(\alph*)]
        \item The axioms in $\Gr\cup Ax_R$ are satisfied because $G$ is the group with presentation $\langle S|R\rangle$. 
        \item If $P(x,x)$ holds for some $x\in G$, then there exists $z_2$ in the monoid generated by $\bar{t}_2$ and the centraliser of $\bar{t}_1$, such that \[z_2 \bar{t}_2 \bar{t}_1 \bar{t}_2^{-1} z_2^{-1} =  z_1^{-1} x z_1=\bar{t}_1.\] In this case, the element $\bar{t}_2^{-1} z_2^{-1}$ is in the centraliser of $\bar{t_1}$. Since a monoid is closed under multiplication by definition, the element $\bar{t}_2^{-1}$ is in the monoid generated by $\bar{t}_2$ and the centraliser of $\bar{t}_1$, which contradicts to the assumption. Thus the the irreflexivity axiom in $\overline{\AxPCL}'$ is satisfied. 
        \item If $P(x,y)$ and $P(y,z)$ holds for $x,y\in G$, then there exist $z_1,z_2,z_3,z_4\in G$, such that \[x=z_1 z_2 \bar{t}_2\bar{t}_1\bar{t}_2^{-1} z_2^{-1} z_1^{-1},\]
        \[y=z_1 \bar{t}_1 z_1^{-1} = z_3 z_4 \bar{t}_2\bar{t}_1\bar{t}_2^{-1} z_4^{-1} z_3^{-1},\]
        \[z=z_3 \bar{t}_1 z_3^{-1},\]
        and $z_2$ and $z_4$ are in the monoid generated by $\bar{t}_2$ and the centraliser of $\bar{t}_1$. By the second equality, the element $\bar{t}_2^{-1} z_4^{-1} z_3^{-1} z_1$ is in the centraliser of $\bar{t}_1$. Since a monoid is closed under multiplication by definition, the element \[z_3^{-1} z_1 z_2 =z_4 \bar{t}_2(\bar{t}_2^{-1} z_4^{-1} z_3^{-1} z_1)z_2\] is in the monoid generated by $\bar{t}_2$ and the centraliser of $\bar{t}_1$. The elements $z_3, z_3^{-1} z_1 z_2 \in G$ satisfy the conditions in the definition of $P(x,z)$, hence $P(x,z)$ holds true. Thus the transitivity axiom in $\overline{\AxPCL}'$ is satisfied.
        \item If $P(x,y)$ holds for some $x,y\in G$, then there exist $z_1, z_2\in G$ such that the conditions in the definition of $P(x,y)$ are satisfied. Then $z z_1, z_2 \in G$ satisfy the conditions in the definition of $P(zxz^{-1},y)$, hence $P(zxz^{-1},y)$. Thus the conjugacy invariance axiom $\PCB$ is satisfied.
        \item The axiom $P((t_2 \cdot t_1)\cdot t_2', t_1)$ is satisfied by taking $z_1=z_2=e$ in the definition of $P$.
    \end{enumerate}
    Therefore the theory $\Gr\cup \Ax_R \cup \overline{\AxPCL}' \cup \PCB \cap\{P((t_2\cdot t_1) \cdot t'_2,t_1)\}$ is consistent if $\bar{t}_2^{-1}$ is not in the monoid generated by $\bar{t}_2$ and the centraliser of $\bar{t}_1$.

\end{proof}

\subsection{Examples}
\begin{example}
    Consider the Fibonacci group $F(2,n)$ $(n\ge 2)$ with the presentation a shown in Example~\ref{ex:Fibonacci}. According to \cite[Theorem 5.2]{motegi2017generalized}, the element $a_0$ is a generalised torsion. To prove this fact automatically when $n=11$ or $n=12$, we can verify the inconsistency (Task 11.1 and Task 11.2) of $\Gr\cup \Ax_R\cup\AxPC'\cup \PB\cup\{P(a_0)\}$ using Prover9, and then apply Proposition~\ref{prop-generalised-torsion}.
\end{example}

\begin{example}
    In Example~\ref{ex:sl2-weak}, we established the absence of bi-invariant circular orders via weakened theories. In fact, the axiom of cyclicity is redundant in the automated proof. By Prover9, we can prove that $\Gr\cup \Ax_R\cup \overline{\AxPCL}'\cup \PCB\cup\{P(a\cdot b,b\cdot a)\}$ is inconsistent (Task 12.1). Thus by Proposition~\ref{prop-CBO-monoid}, the element $a^{-1}$ is in the monoid generated by $a$ and the centraliser of $ba$. Alternatively, this condition follows from the relation $a^{-1}=a^3$.
\end{example}

\begin{example}
In Example~\ref{ex:Poincare}, we established the non-left-orderability of the fundamental group of the Poincar\'e homology sphere via weakened theories. By Proposition~\ref{prop-torsion}, the inconsistency of the weakened theory implies that $a$ is a torsion. Alternatively, it follows from that this group is finite.
\end{example}

\begin{example}
    In Example~\ref{ex:5_2}, we established the non-bi-orderability of the knot group of $5_2$ via weakened theories in two ways (the inequalities $ba\neq ab^2$ and $ba\neq ab$). By Proposition~\ref{prop-generalised-torsion}, it follows that $a^{-1}b^{-1}ab^2$ and $b^{-1}a^{-1}ba$ are generalised torsions. Note that the latter one was previously discovered in \cite[Theorem 7]{naylor2016generalized}.
\end{example}

\begin{example}
    In Example~\ref{ex:weeks}, we established the absence of bi-invariant circular orders via weakened theories. In this example, the axiom of cyclicity is also redundant. We can verify the inconsistency of $\Gr\cup \Ax_R\cup \overline{\AxPCL}'\cup \PCB\cup\{P(a\cdot b,b\cdot a)\}$ (Task 15.1) using Prover9. Thus by Proposition~\ref{prop-CBO-monoid}, the element $a^{-1}$ is in the monoid generated by $a$ and the centraliser of $ba$. Alternatively, this condition follows from the relations $b^{-1} = a(ba)^{-1}$ and $a^{-1}= b^{-3} a b^{-1} a^3b^{-1}$.
\end{example}

\section{Relative convexity and strength of cyclicity axiom}\label{sec:convexity}
In this section, we further explore the consistency of the theory regarding bi-invariant circular orders when the axiom of cyclicity is removed. We do not weaken the axiom of connectedness.

To state our criterion, we introduce the concept of left relative convexity. In contrast to traditional usage (relative convexity of left-orderable groups), we do not restrict the ambient group to a left-orderable group. Thus we adopt the following definition introduced by Antol{\i}n, Dicks, and Sunic \cite{antolin2015left}. The proof of equivalence of definitions could be found in \cite[Lemma 2.1]{antolin2021space}.

\begin{definition}\label{def-rel-convex}
    Let $G$ be a group and $H$ be a subgroup of $G$. We say $H$ is left relatively convex in $G$ when any of the following equivalent conditions hold.
    \begin{enumerate}[label=(\alph*)]
        \item There exists a $G$-invariant order on the left $G$-set $G/H$.
        \item There exists a subsemigroup $P$ of $G$ such that $P\sqcup H\sqcup P^{-1}$ is a partition of $G$, and $H P H\subseteq P$.
    \end{enumerate}
\end{definition}

A total preorder $\le$ on a group $G$ is called a \emph{left total preorder} if it is invariant under multiplication. A subset $S$ of $G$ is called \emph{convex} relative to the left total preorder $\le$ if $x,z\in S$ and $x\le y\le z$ implies $y\in S$. 

Based on the condition (b) in Definition~\ref{def-rel-convex}, we establish the following equivalent conditions for a subgroup to be left relatively convex. By the equivalence (a) $\Leftrightarrow$ (c) in the next proposition, the definition of left relative convexity serves as a left total preorder adaptation of the concept of relative convexity in left-orderable groups.

\begin{proposition}\label{prop-rel-convex-1}
    Let $G$ be a group and $H$ be a subgroup of $G$. The following statements are equivalent.
    \begin{enumerate}[label=(\alph*)]
        \item The subgroup $H$ is left relatively convex in $G$.
        \item The subgroup $H$ is the residue group $\{x\in G: e\le x\le e\}$ of some left total preorder $\le$ on $G$. 
        \item The subgroup $H$ is convex relative to some left total preorder $\le$ on $G$.
    \end{enumerate}
\end{proposition}
\begin{proof}
For the implication (a) $\Rightarrow$ (b), let $P$ be the subsemigroup described in the condition (b) in Definition~\ref{def-rel-convex}. Define the binary relation $\le$ on $G$ by $x\le y$ if and only if $x^{-1}y\in P\cup H$. Then we can check that $\le$ is a left total preorder on $G$: the reflexivity follows from $e\in H$, the transitivity follows from $PP \cup PH\cup HP \subseteq P$ and $HH\subseteq H$, the totality follows from \[(P\cup H)\cup(P\cup H)^{-1} = P\cup H\cup P^{-1}=G,\]
and the left-invariance follows from the definition of $\le$. Now we prove that the residue group of $\le$ is $H$. The inequality $e\le x\le e$ holds if and only if both $x$ and $x^{-1}$ are in $P\cup H$, which holds if and only if \[x\in (P\cup H)\cap(P\cup H)^{-1} = (P\cap P^{-1})\cup H=H.\] 

The implication (b) $\Rightarrow$ (c) holds because the residue group is always convex by definition.

For the implication (c) $\Rightarrow$ (a), suppose that $\le$ is a left total preorder on $G$ relative to which $H$ is convex. Let $P$ denote the set $\{x\in G\setminus H: e\le x\}$. We prove that $P$ satisfies the condition (b) in Definition~\ref{def-rel-convex}.
\begin{enumerate}[label=(\alph*)]
    \item For any $x,y\in P$, by the transitivity and the left-invariance of $\le$, we have $e\le x\le xy$. Then by the convexity of $\le$, we have $xy\in G\setminus H$, hence we have $xy\in P$. Therefore $P$ is a subsemigroup.
    \item Because $P$ is a subsemigroup, $H$ is a subgroup, and $P\cap H=\emptyset$, the subsets $P$, $H$, and $P^{-1}$ are disjoint.
    \item For any $x\in G\setminus H$, by the totality and the left-invariance of $\le$, either $e\le x$ or $e\le x^{-1}$ holds. Since $H$ is a subgroup, either $x\in P$ or $x\in P^{-1}$ holds. Thus we have $G=P\sqcup H\sqcup P^{-1}$.
    \item Suppose that $x\in H$ and $y\in P$. If $xy\in H$, then $y=x^{-1}(xy)\in P\cap H$. If $xy\in P^{-1}$, then $x=(xy)y^{-1}\in H\cap P^{-1}$. Because $G=P\sqcup H\sqcup P^{-1}$, we have $xy\in P$. Thus we have $HP\subseteq P$. For the same reason, we have $PH\subseteq P$. Therefore $HPH\subseteq PH\subseteq P$.
\end{enumerate}
\end{proof}

Next, we establish some equivalent statements for the condition (a) in Definition~\ref{def-rel-convex}. Note that the equivalence (a) $\Leftrightarrow$ (b) in the following proposition generalises \cite[Theorem 1.4.10]{clay2010space}.

\begin{proposition}\label{prop-rel-convex-2}
    Let $G$ be a group and $H$ be a subgroup of $G$. The following statements are equivalent.
    \begin{enumerate}[label=(\alph*)]
        \item The subgroup $H$ is left relatively convex in $G$.
        \item The subgroup $H$ is a kernel of an order preserving $G$-action on some totally ordered set $(\Omega,<)$.
        \item For any finite set of elements $g_1,\ldots, g_n\in G\setminus H$, there exist $\varepsilon_1,\ldots,\varepsilon_n\in\{-1,1\}$ such that the subsemigroup generated by $H g_1^{\varepsilon_1}H,\ldots,Hg_n^{\varepsilon_n}H$ does not contain the identity element $e$.
        \item For any finite set of elements $g_1,\ldots, g_n\in G\setminus H$, there exist $\varepsilon_1,\ldots,\varepsilon_n\in\{-1,1\}$ such that the subsemigroup generated by $g_1^{\varepsilon_1},\ldots,g_n^{\varepsilon_n}$ has empty intersection with $H$.
    \end{enumerate}
\end{proposition}
\begin{proof}
For the implication (a) $\Rightarrow$ (b), suppose that $H$ is left relatively convex in $G$. By the condition (a) in Definition~\ref{def-rel-convex}, there exists a $G$-invariant order $<$ on the left $G$-set $G/H$. And the subgroup $H$ is the kernel of the order preserving $G$-action on the totally ordered set $(G/H,<)$.

The implication (b) $\Rightarrow$ (c) was proved by Tararin; see \cite[Proposition 5.1.5]{kopytov1996right}.

The implication (c) $\Rightarrow$ (d) holds because $H$ is a subgroup: if $h\in H$ is in the subsemigroup generated by $g_1^{\varepsilon_1},\ldots, g_n^{\varepsilon_n}$, then $e$ is in the subsemigroup generated by $g_1^{\varepsilon_1},\ldots, g_n^{\varepsilon_n}$ and $h^{-1} g_1^{\varepsilon_1},\ldots, h^{-1} g_n^{\varepsilon_n}$.

The implication (d) $\Rightarrow$ (a) can be deduced from \cite[Lemma 8]{chiswell1993soluble}; see also \cite[Lemma 2.2.3]{glass1999partially}.
\end{proof}

Finally, we prove that the set of left relatively convex subgroups is closed under arbitrary intersection, generalising \cite[Proposition 5.1.10]{kopytov1996right}, the same property for relatively convex subgroups of a left-orderable group.

\begin{proposition}\label{convex-intersection}
    The intersection of left relatively convex subgroups of $G$ is left relatively convex in $G$.
\end{proposition}
\begin{proof}
Let $\left\{H_\alpha: \alpha \in I\right\}$ be a family of left relatively convex subgroups of $G$. By the implication (a) $\Rightarrow$ (b) in Proposition \ref{prop-rel-convex-2}, there exist totally ordered sets $(\Omega_\alpha,<_\alpha)$ ($\alpha\in I$) such that $H_\alpha$ is the kernel of an order preserving $G$-action on $(\Omega_\alpha,<_\alpha)$. 

Let $\Omega$ denote the disjoint union of all $\Omega_\alpha$ ($\alpha\in I$). We choose an arbitrary total order $<_{\mbox{index}}$ on $I$, then define $x<y$ for some $x\in \Omega_\alpha$ and $y\in \Omega_\beta$ if and only if either $\alpha<_{\mbox{index}}\beta$, or $\alpha=\beta$ and $x<_\alpha y$. In this way, the binary relation $<$ be a total order on $\Omega$ such that each inclusion map $\Omega_\alpha\hookrightarrow \Omega$ is order preserving. Therefore, the natural $G$-action on $\Omega$ is order preserving. Since the kernel of this group action is $\bigcap_{\alpha \in I} H_\alpha$, by the implication (b) $\Rightarrow$ (a) in Proposition \ref{prop-rel-convex-2}, the statement holds true.
\end{proof}

Now we state the criterion for the existence of a predicate satisfying all axioms for a bi-invariant circular order, except for the axiom of cyclicity.

\begin{proposition}
    Let $\overline{\AxPCL}$ denote $\AxPCL$ minus the axiom of cyclicity. There exists a binary relation $P(\cdot,\cdot)$ on the group $G$ satisfying the axioms in $\overline{\AxPCL}\cup\PCB$ if and only if the centraliser of each subset of $G$ is left relatively convex in $G$.
\end{proposition}
\begin{proof}
    We first prove the ``only if'' part. Suppose that there exists a binary relation $P(\cdot,\cdot)$ on $G$ satisfying the axioms in $\overline{\AxPCL}\cup\PCB$. For each $x\in G$, we define a binary relation $\le_x$ on $G$ by $y\le_x z$ if and only if either $y^{-1}z$ commutes with $x$, or $P(yxy^{-1},zxz^{-1})$ holds true. We prove that $\le_x$ is a left total preorder on $G$ with the residue being the centraliser of $x$.
    \begin{enumerate}[label=(\alph*)]
        \item The reflexivity follows from that $e$ commutes with $x$.
        \item Suppose that $y\le_x z$ and $z\le_x u$. If both $y^{-1}z$ and $z^{-1} u$ commute with $x$, then $y^{-1} u$ also commutes with $x$. We consider four scenarios. If $y^{-1}z$ commutes with $x$ and $P(zxz^{-1},uxu^{-1})$ holds, then by $zxz^{-1}= yxy^{-1}$, we have $P(yxy^{-1},uxu^{-1})$. If $P(yxy^{-1},zxz^{-1})$ holds and $z^{-1}u$ commutes with $x$ then by $zxz^{-1}= uxu^{-1}$, we also have $P(yxy^{-1},uxu^{-1})$. If both $P(yxy^{-1},zxz^{-1})$ and $P(zxz^{-1},uxu^{-1})$ hold true, then by the transitivity axiom in $\overline{\AxPCL}$, we have $P(yxy^{-1},uxu^{-1})$. In either way, we have $z\le_x u$. Thus we proved the transitivity of $\le_x$.
        \item The totality follows from the connectedness axiom in $\overline{\AxPCL}$.
        \item The left-invariance follows from the conjugacy invariance axiom $\PCB$. 
        \item If $y$ is in the centraliser of $x$, then so is $y^{-1}$. By definition, we have $e\le_x y\le_x e$. So the centraliser of $x$ is a subgroup of the residue of $\le_x$.
        \item If $y$ is not in the centraliser of $x$, then neither is $y^{-1}$. If we also have $e\le_x y \le_x e$, then by definition, we have $P(x,yxy^{-1})$ and $P(yxy^{-1},x)$. By the transitivity axiom in $\overline{\AxPCL}$, we have $P(x,x)$, which contradicts to the irreflexivity axiom in $\overline{\AxPCL}$. Thus the residue of $\le_x$ is a subgroup of the centraliser of $x$.
    \end{enumerate}
    By the implication (b) $\Rightarrow$ (a) in Proposition~\ref{prop-rel-convex-1}, the centraliser of any group element is left relatively convex in $G$. By Proposition~\ref{convex-intersection}, the centraliser of any subset is left relatively convex in $G$.
    
    Then we prove the ``if'' part. Suppose that the centraliser of any $x\in G$ is left relatively convex in $G$. Let $\mathrm{Con}(G)$ denote the set of conjugacy classes of $G$. For each conjugacy class $\gamma\in \mathrm{Con}(G)$, select an group element $x_\gamma$ in $\gamma$. By the implication (a) $\Rightarrow$ (b) in Proposition~\ref{prop-rel-convex-1}, there exists a left total preorder $\le_\gamma$ on $G$ such that the residue group of $\le_\gamma$ is the centraliser of $x_\gamma$. We choose an arbitrary total order $<_{\mbox{index}}$ on $\mathrm{Con}(G)$, then define $P(y,z)$ for some $y\in \gamma_1$ and $z\in \gamma_2$ ($\gamma_1, \gamma_2\in \mathrm{Con}$) if and only if either $\gamma_1<_{\mbox{index}}\gamma_2$, or $\gamma_1=\gamma_2$ and the following four conditions hold for some pair $(y_0, z_0)\in G\times G$:
    \begin{enumerate}[label=(\alph*)]
        \item $y= y_0 x_{\gamma_1} y_0^{-1}$.
        \item $z= z_0 x_{\gamma_1} z_0^{-1}$.
        \item $y_0\le_{\gamma_1} z_0$ holds.
        \item $z_0\le_{\gamma_1} y_0$ does not hold.
    \end{enumerate}
    We prove that the binary relation $P(\cdot,\cdot)$ satisfies the axioms in $\overline{\AxPCL}\cup\PCB$.
   \begin{enumerate}[label=(\alph*)]
        \item We first prove the irreflexivity axiom in $\overline{\AxPCL}$. Suppose that $P(y,y)$ for some $y\in G$. By definition, there exists $(y_0,z_0)\in G\times G$ satisfying the four conditions above. Let $\gamma$ denote the conjugacy class of $y$, then by $y= y_0 x_{\gamma} y_0^{-1}=z_0 x_{\gamma} z_0^{-1}$, the element $y_0^{-1} z_0$ is in the centraliser of $x_\gamma$. In this case, we have $z_0\le_{\gamma_1} y_0$, which contradicts to the fourth condition.
        \item Then we prove the transitivity axiom in $\overline{\AxPCL}$. Suppose that both $P(y,z)$ and $P(z,u)$ hold true. Let $\gamma_1, \gamma_2,\gamma_3$ denote the conjugacy classes of $y,z,u$ respectively. By definition, we have $\gamma_1\le_{\mbox{index}}\gamma_2$ and $\gamma_2\le_{\mbox{index}}\gamma_3$. By the transitivity of $<_{\mbox{index}}$, we have $\gamma_1\le_{\mbox{index}}\gamma_3$. If $\gamma_1<_{\mbox{index}}\gamma_3$, then $P(y,u)$ holds true. Otherwise, we have $\gamma_1=\gamma_2=\gamma_3$, and there exist pairs $(y_0,z_0)$ and $(y_1,z_1)\in G\times G$ satisfying the four conditions for $(y,z)$ and $(z,u)$ respectively. We prove that the pair $(y_0,z_1)$ satisfies the four conditions for $(y,u)$. The first two conditions follow from corresponding conditions for $(y,z)$ and $(z,u)$. By $z= z_0 x_{\gamma_1} z_0^{-1}= y_1x_{\gamma_1} y_1^{-1}$, the element $y_1^{-1} z_0$ is in the centraliser of $x_{\gamma_1}$. Thus we have $z_0\le_{\gamma_1} y_1$. By the transitivity of $\le_{\gamma_1}$, we have $y_0 \le_{\gamma_1} z_0\le_{\gamma_1} y_1\le_{\gamma_1} z_1$. Thus the third condition is satisfied. If $z_1\le_{\gamma_1}y_0$ holds, then by the transitivity of $\le_{\gamma_1}$, we have $z_0\le_{\gamma_1} y_1\le_{\gamma_1} z_1\le_{\gamma_1} y_0$, which contradicts to the corresponding conditions for $(y,z)$. Therefore, we have $P(y,u)$ if both $P(y,z)$ and $P(z,u)$ hold true,
        \item Now we prove the connectedness axiom in $\overline{\AxPCL}$. Suppose that $y$ and $z$ are two distinct group elements. Let $\gamma_1$ and $\gamma_2$ denote the conjugacy classes of $y$ and $z$ respectively. If $\gamma_1\neq \gamma_2$, then by the connectedness of $<_{\mbox{index}}$, we have either $\gamma_1<_{\mbox{index}}\gamma_2$ or $\gamma_2<_{\mbox{index}}\gamma_2$, thus we have either $P(y,z)$ or $P(z,y)$. Otherwise, suppose that $\gamma_1= \gamma_2$, then there exists $(y_0,z_0)\in G\times G$ such that $y=y_0 x_{\gamma_1}y_0^{-1}$ and $z=z_0 x_{\gamma_1}z_0^{-1}$. By the totality of $\le_{\gamma_1}$, we have either $y_0\le_{\gamma_1} z_0$ or $z_0\le_{\gamma_1} y_0$. By $y\neq z$, the element $y_0^{-1}z_0$ is not in the centraliser of $x_{\gamma_1}$, thus only one of $y_0\le_{\gamma_1} z_0$ and $z_0\le_{\gamma_1} y_0$ holds.
        \item Finally we prove the conjugacy invariance axiom $\PCB$. Suppose that $P(y,z)$ holds true. Let $\gamma_1$ and $\gamma_2$ denote the conjugacy classes of $y$ and $z$ respectively. Then by definition, either $\gamma_1<_{\mbox{index}} \gamma_2$, or $\gamma_1=\gamma_2$ and there exists $(y_0,z_0)\in G\times G$ satisfying the four conditions for $(y,z)$. For any $u\in G$, we have $uyu^{-1}\in\gamma_1$ and $uzu^{-1}\in \gamma_2$. If $\gamma_1<_{\mbox{index}} \gamma_2$, then we have $P(uyu^{-1},uzu^{-1})$. Otherwise, we prove that $(u y_0,u z_0)\in G\times G$ satisfies the four conditions for $(uyu^{-1}, u z u^{-1})$. The first two conditions follow from corresponding conditions for $(y,z)$. The last two conditions follow from corresponding conditions for $(y,z)$ and the left-invariance of $\le_{\gamma_1}$. Therefore $P(y,z)$ implies $P(uyu^{-1},uzu^{-1})$.
    \end{enumerate}
\end{proof}

\section{Absolute cofinality and non-left-orderability}\label{sec:cofinality}
In this section, we introduce a methodology to integrate the fixed point method for non-left-orderability into automated reasoning. The term ``fixed point method'' originates from the dynamic realisation of a left order. However, instead of introducing the dynamic realisation, we opt for the concept of cofinal elements for convenience, which essentially yields the same proofs.

We define the left absolute cofinality as a dual concept of the left relative convexity. This makes our definition slightly different from traditional usage.

To begin with, we introduce the concept of the left relatively convex subgroup closure and integrate it into automated reasoning.

\subsection{Relatively convex subgroup closure}

By Proposition~\ref{convex-intersection}, the left relatively convex subgroups form a Moore collection. So we can define a natural closure operator $\mathrm{cl}$ based on this family. The following definition extends the definition of the relatively convex subgroup closure in \cite{longobardi2000right} to arbitrary groups.

\begin{definition}\label{def-closure}
    For a subset $A$ in a group $G$, we define the left relatively convex subgroup closure $\mathrm{cl}(A)$ as the intersection of all left relatively convex subgroups $H$ with $A\subseteq H \subseteq G$.
\end{definition}

By Proposition~\ref{convex-intersection}, a subset $A\subseteq G$ is a left relatively convex subgroup of $G$ if and only if $\mathrm{cl}(A)=A$.

\begin{proposition}\label{prop-equiv-cofinal}
    Let $A$ be a subset of the group $G$, and $g\in G$ be a group element. Let $\langle A\rangle$ denote the subgroup generated by $A$. The following statements are equivalent.
    \begin{enumerate}[label=(\alph*)]
        \item The element $g$ is in the left relatively convex subgroup closure $\mathrm{cl}(A)$.
        \item For any subsemigroup $S$ with $A\cup A^{-1}\subseteq S\subseteq G$ and $S\cup S^{-1}=G$, we have $g\in S$.
        \item For any left total preorder $\le$ on $G$, there exists $x\in \langle A\rangle$ such that $g\le x$.
    \end{enumerate}
\end{proposition}
\begin{proof}
    First we prove the implication (a) $\Rightarrow$ (b). Suppose that $S$ is a subsemigroup $S$ with $A\cup A^{-1}\subseteq S\subseteq G$ and $S\cup S^{-1}=G$. Let $P$ denote the set $S\setminus S^{-1}$. By $S\cup S^{-1}=G$, the group $G$ can be written as the disjoint union $P\sqcup (S\cap S^{-1})\sqcup P^{-1}$. We first prove that $P$ is a subsemigroup. For any $x,y\in P$, we have $xy\in S$ by the closure of $S$ under multiplication. If we have $xy\in S^{-1}$, then $y^{-1}=(xy)^{-1}x$ is in $S$ by the closure of $S$ under multiplication, which contradicts to the assumption that $y\in P$. Now we prove that $(S\cap S^{-1})P(S\cap S^{-1})\subseteq P$. On the one hand, since $P\subseteq S$, by the closure of $S$ under multiplication, we have $(S\cap S^{-1})P(S\cap S^{-1})\subseteq S$. On the other hand, if there exist $x,z\in S\cap S^{-1}$ and $y\in P$ such that $xyz\in S^{-1}$, then $y^{-1}=z(xyz)^{-1}x$ is in $S$ by the closure of $S$ under multiplication, which contradicts to the assumption that $y\in P$. So we have $(S\cap S^{-1})P(S\cap S^{-1})\subseteq P$ by the definition of $P$. Therefore $S\cap S^{-1}$ is left relatively convex by the condition (b) in Definition~\ref{def-rel-convex}. By the assumption that $g\in\mathrm{cl}(A)$, we have \[g\in\mathrm{cl}(A)\subseteq \mathrm{cl}(S\cap S^{-1})= S\cap S^{-1} \subseteq S.\]

    Then we prove the implication (b) $\Rightarrow$ (c). For any left total preorder $\le$ on $G$, consider the set \[S:=\{x\in G: \mbox{for all } a_1 \in \langle A\rangle\mbox{ there exists } a_2\in \langle A\rangle  \mbox{ such that } x a_1\le a_2\}.\]
    We prove that $S$ satisfies the condition described in (b). First, for any $x\in A\cup A^{-1}$ and any $a_1\in \langle A\rangle$, we can choose $a_2 = x a_1 \in \langle A\rangle$. Thus by the reflexivity of $\le$, we have $A\cup A^{-1}\subseteq M$. Second, if $x, y\in S$, then for all $a_1\in \langle A\rangle$ there exists $a_2\in \langle A\rangle$ such that $y a_1\le a_2$, and also there exists $a_3 \in \langle A\rangle$ such that $x a_2 \le a_3$. By the left-invariance and the transitivity of $\le$, we have $xya_1\le x a_2\le a_3$. Thus we have $xy\in S$ by the definition of $S$. In other words, $S$ is a subsemigroup. Finally, for any $x\in G$, if both $x\not\in S$ and $x^{-1}\not \in S$ hold true, then there exists $a_1\in \langle A\rangle$ such that $\lnot(x a_1\le a_2)$ for all $a_2\in \langle A\rangle$, and there exists $a_3\in \langle A\rangle$ such that $\lnot(x^{-1} a_3\le a_4)$ for all $a_4\in \langle A\rangle$. In particular, we have $\lnot(x a_1\le a_3)$ and $\lnot(x^{-1} a_3\le a_1)$. By the left-invariance of $\le$, neither $x a_1\le a_3$ nor $a_3\le x a_1$ holds true, which contradicts to the totality of $\le$. Thus we have $S\cup S^{-1}=G$. Suppose that (b) holds, then we have $g\in S$. By taking $a_1=e$ in the definition of $S$, there exists $x\in \langle A\rangle$ such that $g\le x$. 

    Finally we prove the implication (c) $\Rightarrow$ (a). Suppose that $H$ is an arbitrary left relatively convex subgroup with $A\subseteq H\subseteq G$. Since $H$ is a subgroup, we have $\langle A\rangle \subseteq H$. By the implication (a) $\Rightarrow$ (b) in Proposition~\ref{prop-rel-convex-1}, there exists a left total preorder $\le$ on $G$ with the residue being $H$. Suppose that (c) holds, then there exists $x\in \langle A\rangle$ such that $g\le x$. Define $x\le_{op} y$ if and only if $y\le x$, then we can check that $\le_{op}$ is also a left total preorder. By the statement (c), there exists $y\in \langle A\rangle$ such that $g\le_{op} y$. Because $e\le y\le g\le x\le e$, we have $g\in H$. By the arbitrariness of $H$, we have $g\in \mathrm{cl}(A)$.
\end{proof}

Now we present a methodology for proving $g_{k+1}\in \mathrm{cl}(\{g_1,\ldots, g_k\})$ in a group $G$ using generic automated theorem proving based on the implication (b) $\Rightarrow$ (a) in Proposition~\ref{prop-equiv-cofinal}. For convenience, let $\CC$ denote the following axioms for a unary predicate $P(\cdot)$:
\begin{enumerate} [label=(\alph*)]
\item $\forall x \forall y (P(x)\land P(y)\to   P(x\cdot y))$, \hfill\emph{(closure)} 
\item $\forall x (P(x)\lor P(x'))$.  \hfill\emph{(connectedness)} 
\end{enumerate}

\begin{proposition}\label{prop-cl}
    Let $G$ be a group with presentation $\langle S|R\rangle$. Let $t_1, \ldots, t_k, t_{k+1}$ be ground terms representing group elements $\bar{t}_1, \ldots, \bar{t}_k, \bar{t}_{k+1}$ in $G$ respectively. Then the theory $\Gr \cup \Ax_R\cup \CC\cup \{P(t_i)\land P(t_i'): i=1,\ldots, k\} \cup \{\lnot P(t_{k+1})\}$ is inconsistent if and only if $\bar{t}_{k+1}$ is in the left relatively convex subgroup closure $\mathrm{cl}(\{\bar{t}_1,\ldots,\bar{t}_k\})$.
\end{proposition}
\begin{proof}
    We first prove the ``only if'' part. Assume that $\bar{t}_{k+1}$ is not in the left relatively convex subgroup closure $\mathrm{cl}(\{\bar{t}_1,\ldots,\bar{t}_k\})$. By the implication (b) $\Rightarrow$ (a) in Proposition~\ref{prop-equiv-cofinal}, there exists a subsemigroup $S_+\subseteq G$ with $\bar{t}_i,\bar{t}_i^{-1}\in S_+$ ($1\le i\le k$), $S_+\cup S_+^{-1}=G$, and $\bar{t}_{k+1}\not\in S_+$. Define the unary predicate $P$ by $P(t)$ if and only if $t$ represents an element in $S_+$. Then we can check that the group $G$ together with the predicate $P$ constitutes a model for the theory $\Gr\cup \Ax_R\cup\CC\cup \{P(t_i)\land P(t_i'): i=1,\ldots, k\} \cup \{\lnot P(t_{k+1})\}$. Therefore it is consistent if $\bar{t}_{k+1}$ is not in $\mathrm{cl}(\{\bar{t}_1,\ldots,\bar{t}_k\})$.

    Then we prove the ``if'' part. Assume that the theory $\Gr \cup \Ax_R\cup \CC\cup \{P(t_i)\land P(t_i'): i=1,\ldots, k\} \cup \{\lnot P(t_{k+1})\}$ is consistent, then there is a model $\mathfrak{M}$ of it. Let $S_+\subseteq G$ denote the subset \[S_+:=\{x\in G: x\mbox{ is represented by a ground term }t\mbox{ such that } P(t) \mbox{ in } \mathfrak{M}\}\]
    Then $S_+$ has the following properties.
    \begin{enumerate}[label=(\alph*)]
        \item For any $x,y\in S_+$, there exist ground terms $t_1, t_2$ representing $x,y$ respectively such that $P(t_1)$ and $P(t_2)$ hold true in $\mathfrak{M}$. By the closure axiom in $\CC$, $P(t_1\cdot t_2)$ holds true in $\mathfrak{M}$. Since $xy$ is represented by $t_1\cdot t_2$, we have $xy\in S_+$ for any $x,y\in S_+$. Thus $S_+$ is a subsemigroup.
        \item For each $1\le i\le k$, the elements $\bar{t}_i$ and $\bar{t}_i^{-1}$ in $G$ are represented by the ground terms $t_i$ and $t_i'$ respectively. Thus by $P(t_i)\land P(t_i')$, we have $\bar{t}_i, \bar{t}_i^{-1}\in S_+$.
        \item For any $x\in G$, let $t$ be a ground term representing $x$. Then $x^{-1}$ is represented by the ground term $t'$. By the connectedness axiom in $\CC$, either $P(t)$ or $P(t')$ holds true in $\mathfrak{M}$. Thus we have $S_+\cup S_+^{-1}=G$. 
    \end{enumerate}
    For any ground term $s$ representing $\bar{t}_{k+1}$ in $G$, by the definition of $\langle S|R\rangle$, the axioms $\Gr\cup \Ax_R$ imply that $s=t_{k+1}$ in $\mathfrak{M}$. By the axiom $\neg P(t_{k+1})$, we have $\neg P(s)$ in $\mathfrak{M}$. Thus we have $\bar{t}_{k+1}\not\in S_+$.

    By the implication (a) $\Rightarrow$ (b) in Proposition~\ref{prop-equiv-cofinal}, the element $\bar{t}_{k+1}$ is not in the left relatively convex subgroup closure $\mathrm{cl}(\{\bar{t}_1,\ldots, \bar{t}_{k}\})$.
\end{proof}

\subsection{Absolute cofinality}

We say a subgroup $H\subseteq G$ is \emph{left absolutely cofinal} if the left relatively convex subgroup closure $\mathrm{cl}(H)$ equals to $G$. By the equivalence (a) $\Leftrightarrow$ (c) in Proposition~\ref{prop-equiv-cofinal}, a subgroup $H\subseteq G$ is left absolutely cofinal if and only if it is cofinal\footnote{A subset $B\subseteq A$ of a preordered set $(A,\le)$ is called \emph{cofinal} if for every $a\in A$ there exists $b\in B$ such that $a\le b$.} with respect to every left total preorder $\le$ on $G$.

It follows from the condition (a) in Definition~\ref{def-rel-convex} that the trivial subgroup $\{e\}$ is left relatively convex if and only if $G$ is left-orderable. Now we establish the left absolutely cofinal counterpart of this fact.

\begin{proposition}\label{proposition-e-cofinal}
    The trivial subgroup $\{e\}$ is not left absolutely cofinal in the group $G$ if and only if $G$ admits a nontrivial left-orderable quotient.
\end{proposition}
\begin{proof}
    We first prove the ``only if'' part. If $e$ is not left absolutely cofinal in the group $G$, then there exists a proper subgroup $H$ which is left relatively convex in $G$. By the implication (a) $\Rightarrow$ (d) in Proposition~\ref{prop-rel-convex-2}, for any finite set of elements $g_1,\ldots, g_n\in G\setminus H$, there exist $\varepsilon_1,\ldots,\varepsilon_n\in\{-1,1\}$ such that the subsemigroup generated by $g_1^{\varepsilon_1},\ldots,g_n^{\varepsilon_n}$ has empty intersection with $H$. By \cite[Lemma 2.2.3]{glass1999partially}, the quotient group $G/\mathrm{core}(H)$ is left-orderable, where $\mathrm{core}(H):=\cap\{g H g^{-1}:g\in G\}$ is the largest normal subgroup of $G$ contained in $H$. Since $H$ is proper, the quotient group $G/\mathrm{core}(H)$ is nontrivial.

    Then we prove the ``if'' part. Suppose that $N$ is a proper normal subgroup of $G$ such that $G/N$ is left-orderable. The $G$-action on $G/N$ preserves the left orders, so by the condition (a) in Definition~\ref{def-rel-convex}, the subgroup $N$ is left relatively convex in $G$. Therefore $e$ is not left absolutely cofinal in $G$.
\end{proof}

It is worth noting that the condition that a group does not admit any nontrivial left-orderable quotients holds significant importance in topology. According to the L-space conjecture, it is conjectured that a closed connected $3$-manifold is an L-space if and only if its fundamental group does not admit any nontrivial left-orderable quotient. If we replace this condition with the non-left-orderability, we would have to assume the manifold is at least irreducible. This would only make things more complicated; compare \cite[Theorem 1.9]{boyer2022order} to \cite[Corollary 1.10]{boyer2022order} for example.

While one may establish the non-left-orderability by proving $\mathrm{cl}(e)=G$ through successive applications of Proposition~\ref{prop-cl}, this method is not much different from the positive cone formalisation presented in Proposition~\ref{prop:positive_cone} and Proposition~\ref{prop:positive_cone_C}. In the remainder of this subsection, we introduce an alternative approach that simplifies the computation.

It is well-known that the absolute cofinality of a cyclic subgroup $\langle g \rangle$ implies that every conjugate of $g$ has the same sign with respect to any given left order on $G$; see \cite[Property 3.1]{conrad1959right} or \cite[Lemma 4.6]{boyer2017foliations} for example. In order to generalise this fact to left total preorders, we prove the following statement.

\begin{proposition}\label{prop-isolated}
    Let $\le$ be a left total preorder on the group $G$. If $e \le x^n$ for some $x\in G$ and some positive integer $n$, then $e\le x$.
\end{proposition}
\begin{proof}
    By the connectedness of $\le$, we have either $e\le x$ or $x\le e$. In the second case, we have \[e\le x^n\le x^{n-1}\le \cdots \le x\] by the transitivity and the left-invariance of $\le$. Therefore, in either case, we have $e\le x$.
\end{proof}

Now we generalise \cite[Lemma 4.6]{boyer2017foliations} to left total preorders.

\begin{proposition}\label{prop-cofinal-conjugation}
    Let $\le$ be a left total preorder on the group $G$. Let $x\in G$ be an element with $e\le x$ and $\langle x\rangle$ being left absolutely cofinal. Then we have $e\le y x y^{-1}$ for every $y\in G$.
\end{proposition}
\begin{proof}
    By the connectedness of $\le$, we have either $y\le e$ or $e\le y$. We consider two scenarios separately.

    First, suppose that $y\le e$. Then by the cofinality of $\langle x\rangle$ with respect to $\le $, there exists an integer $n$ such that $y^{-1}\le x^n$. Because \[y^{-1}\le x^n\le x^{n+1}\le x^{n+2}\le \cdots,\] we can assume that $n$ is positive. And we have \[e \le yx^n \le yx^n y^{-1}=(y x y^{-1})^n.\]
    By Proposition~\ref{prop-isolated}, we have $e\le yxy^{-1}$.

    Then, suppose instead that $e \le y$. Define the binary relation $\le_{op}$ by $x\le_{op}y$ if and only if $y\le x$, then we can check that $\le_{op}$ is also a left total order on $G$. By the cofinality of $\langle x \rangle$ with respect to $\le_{op}$, there exists an integer $n$ such that $y^{-1}\le_{op} x^{-n}$. Because \[\cdots\le x^{-n-2}\le x^{-n-1}\le x^{-n}\le y^{-1},\]
    we can also assume that $n$ is positive. And we have 
    \[e\le y \le y x^n y^{-1}= (yxy^{-1})^n.\]
    By Proposition~\ref{prop-isolated}, we have $e\le yxy^{-1}$.
\end{proof}

A subsemigroup $A$ of the group $G$ is called \emph{isolated} if $x^n\in A$ for some positive integer $n$ implies $x\in A$ for every $x\in G$. It is called \emph{normal} if $xAx^{-1}\subseteq A$ for every $x\in G$. The following statement is a left total preorder adaptation of the technique developed in \cite[Section 3]{nie20201}.

\begin{proposition}\label{prop-stage-2}
    Let $g_0, g_1, \ldots, g_k\in G$. Suppose that $\langle g_i\rangle$ is left absolutely cofinal in $G$ for each $0\le i\le k$, and that $\{e\}$ is not left absolutely cofinal in $G$. Then there exists an isolated normal subsemigroup $S$ such that:
    \begin{enumerate} [label=(\alph*)]
        \item for each $1\le i\le k$, either $g_i\in S$ or $g_i^{-1}\in S$ holds;
        \item the element $g_0$ is in $S$;
        \item the identity element $e$ is not in $S$.
    \end{enumerate}
\end{proposition}
\begin{proof}
    Since $\{e\}$ is not left absolutely cofinal in $G$, there exists a proper subgroup $H$ which is left relatively cofinal in $G$. By the implication (a) $\Rightarrow$ (b) in Proposition~\ref{prop-rel-convex-1}, there exists a left total preorder $\le$ on $G$ with the residue being $H$. We assume $e\le g_0$ without loss of generality, because otherwise we can replace $\le$ with $\le_{op}$, where $x\le_{op} y$ if and only if $y\le x$. 
    
    For any $g\in G$, we define the binary relation $\le_{g}$ by $x \le_g y$ if and only if $x g \le y g$. Then the reflexivity, transitivity, totality, and left-invariance of $\le$ imply the same properties for $\le_g$ respectively. Therefore $\le_g$ is a left total order on $G$. Moreover, the residue group of $\le_g$ is $g^{-1}Hg$. 

    Let $P_g$ denote the set $\{x\in G\setminus g^{-1} H g: e\le_g x\}$. Then by the proof of the implication (b) $\Rightarrow$ (a) in Proposition~\ref{prop-rel-convex-1}, the set $P_g$ is a subsemigroup of $G$ such that $P_g\sqcup g^{-1} H g\sqcup P_g^{-1}$ is a partition of $G$, and $H P_g H\subseteq P_g$. We prove that the subset $S:= \bigcap_{g\in G} P_g$ satisfies the desired conditions.
    \begin{enumerate}[label=(\alph*)]
        \item $S$ is a subsemigroup, because every $P_g$ is a subsemigroup.
        \item If $x^n\in S$ for some $x\in G$ and some positive integer $n$, then we have $e\le_g x^n$ and $x^n\not\in g^{-1} H g$ for every $g\in G$. By Proposition~\ref{prop-isolated}, we have $e\le_g x$ for every $g\in G$. Since $g^{-1}H g$ is a subgroup, we have $x\not\in g^{-1} H g$ for any $g\in G$. Thus we have $x\in S$. In other words, $S$ is isolated.
        \item By the equation $x P_g x^{-1}= P_{xg}$, the subsemigroup $S$ is normal.
        \item For each $0\le i\le k$, since $\langle g_i\rangle$ is left absolutely cofinal in $G$, by the implication (b) $\Rightarrow$ (a) in Proposition~\ref{prop-rel-convex-1}, we have $g_i\not \in g^{-1}H g$ for any $g\in G$. By $g_i\not\in H$, there exist exponents $\varepsilon_i\in\{-1,1\}$ such that $g_i^{\varepsilon_i}\in P_e$. By $e\le g_0$, we have $\varepsilon_0=1$. Since $g^{-1}H g$ is a subgroup, we have $g_i^{\varepsilon_i}\not \in g^{-1}H g$ for any $g\in G$. Since $\langle g_i^{\varepsilon_i}\rangle=\langle g_i\rangle$ is left absolutely cofinal in $G$, by Proposition~\ref{prop-cofinal-conjugation}, we have $e\le_g g_i^{\varepsilon_i}$ for any $g\in G$. Therefore, we have $g_0\in S$, and $g_i^{\varepsilon_i}\in S$ for each $1\le i \le k$.
        \item Since $e\in g^{-1}H g$ for every $g\in G$, we have $e\not\in S$.
    \end{enumerate}
\end{proof}

Now we present our methodology to integrate the fixed point method for non-left-orderability into automated reasoning. We suppose that the left absolute cofinality was checked through successive uses of Proposition~\ref{prop-cl}. Then we may use Proposition~\ref{prop-fixed-point} below to establish the non-left-orderability.

For each positive integer $m$ and each ground term $t$, recursively define the ground term $t^m$ by $t^1=t$ and $t^m=t^{m-1}\cdot t$ for $m\ge 2$. Let $M$ be a set of positive integers. Let $\Isolated(M)$ denote the set of axioms for a unary predicate $P(\cdot)$, containing the following axiom for each $m\in M$:
\begin{enumerate} [label=(\alph*)]
\item $\forall x (P(x^m)\to P(x))$. \hfill\emph{($m$-isolation)} 
\end{enumerate}
\begin{proposition}\label{prop-fixed-point}

Let $G$ be a group with presentation $\langle S|R \rangle$. Let $t_0, t_1,\ldots,t_k$ be ground terms representing group elements $\bar{t}_0,\bar{t}_1,\ldots, \bar{t}_k$ in $G$ respectively. Suppose that $\langle \bar{t}_i\rangle$ is left absolutely cofinal in $G$ for each $0\le i\le k$. Let $\AxPC'$ denote $\AxPC$ minus the axiom of connectedness. Let $M$ be a set of positive integers. If the theory $\Gr\cup\Ax_R\cup\AxPC' \cup \PB\cup \Isolated(M)\cup  \{P(t_0)\} \cup\{P(t_i)\lor P(t_i'): i=1,\ldots, k\}$ is inconsistent, then $G$ does not admit nontrivial left-orderable quotients.
\end{proposition}
\begin{proof}
    If $G$ admits a nontrivial left-orderable quotient, then by Proposition~\ref{proposition-e-cofinal}, the trivial subgroup $\{e\}$ is not left absolutely cofinal in $G$. Then by Proposition~\ref{prop-stage-2}, there exists an isolated normal subsemigroup $S_+$ such that:
    \begin{enumerate} [label=(\alph*)]
        \item for each $1\le i\le k$, either $\bar{t}_i\in S_+$ or $\bar{t}_i^{-1}\in S_+$ holds;
        \item the element $\bar{t}_0$ is in $S_+$;
        \item the identity element $e$ is not in $S_+$.
    \end{enumerate}
    Define the unary predicate $P$ by $P(t)$ if and only if $t$ represents an element in $S_+$. Then we can check that the group $G$ together with the predicate $P$ constitutes a model for the theory  $\Gr\cup\Ax_R\cup\AxPC' \cup \PB\cup \Isolated(M)\cup\{P(t_0)\}\cup \{P(t_i)\lor P(t_i'): i=1,\ldots, k\}$. Therefore the theory is consistent if $G$ admits a nontrivial left-orderable quotient.
\end{proof}

Notice that the set of axioms $\{P(t_0)\} \cup\{P(t_i)\lor P(t_i'): i=1,\ldots, k\}$ in Proposition~\ref{prop-fixed-point} has the same form as $\mathfrak{S}^P$ defined in Proposition~\ref{prop:positive_cone}. Thus, for convenience, we name this set as $\mathfrak{S}^P$ with respect to the pairs $(e, t_i)$ ($i=0,1,\ldots k$).

\subsection{Examples}
\begin{example}
    Consider the Weeks manifold group with the presentation as shown in Example~\ref{ex:weeks}. We prove that this group does not admit nontrivial left-orderable quotients using automated reasoning.

    First, we prove that $\Gr\cup \Ax_R\cup\CC \cup\{P(a)\land P(a'), \neg P(b)\}$ is inconsistent (Task 16.1) using Prover9. By Proposition~\ref{prop-cl}, we have $b\in\mathrm{cl}(a)$. Thus $\langle a\rangle$ is left absolutely cofinal. By symmetry, the subgroup $\langle b\rangle$ is also left absolutely cofinal.

    Consider the set $\mathfrak{S}^P$ with respect to the pairs $(e,a)$ and $(e,b)$. We prove that $\Gr\cup\Ax_R\cup\AxPC' \cup \PB\cup \mathfrak{S}^P$ is inconsistent (Task 16.2) using Prover9. By Proposition~\ref{prop-fixed-point}, the group of interest does not admit nontrivial left-orderable quotients. The non-left-orderability follows from the nontriviality.
\end{example}
\begin{example}

Consider the fundamental group of the fourfold branched cover of the two-bridge knot $K_{[6,-6]}$. According to \cite[Theorem 1.9]{gordon2014taut}, this group is not left-orderable. We consider the following presentation of the group:
\[\langle a_i,b_i \;|\; a_i^3b_i=a_{i+1}^3, a_i b_i^3= b_{i-1}^3 \mbox{ for } i=0,1,2,3\rangle,\]
where the indices are taken modulo $4$.

We can derive contradictions (Task 17.1 and Task 17.2) from the theories $\Gr \cup \Ax_R\cup \CC\cup\{P(a_0)\land P(a_0'),\neg P(b_0)\}$ and $\Gr \cup \Ax_R\cup \CC\cup\{P(b_0)\land P(b_0'),\neg P(a_1)\}$ using Prover9. By Proposition~\ref{prop-cl}, we obtain $b_0\in\mathrm{cl}(a_0)$ and $a_1\in\mathrm{cl}(b_0)$. By symmetry, we have $b_i\in \mathrm{cl}(a_i)$ and $a_{i+1}\in\mathrm{cl}(b_i)$ for each $i=0,1,2,3$. Therefore, each $\langle a_i\rangle$ ($i=0,1,2,3$) is left absolutely cofinal.

Consider the set $\mathfrak{S}^P$ with respect to the pairs $(e,a_i)$ ($i=0,1,2,3$). We prove that the theory $\Gr\cup \Ax_R \cup \AxPC'\cup \PB \cup \mathfrak{S}^P$ is inconsistent (Task 17.3) using Prover9. By Proposition~\ref{prop-fixed-point}, the group does not admit nontrivial left-orderable quotients. The non-left-orderability follows from the nontriviality, which can be verified by building a model (Task 17.4) of $\Gr\cup \Ax_R\cup \{e\neq a_0\}$ using Mace4.
\end{example}

\begin{example}\label{ex:Hyde}
   Hyde \cite{hyde2019group} proved that the group $\mathrm{Homeo}(D,\partial D)$ of homeomorphisms of the disc that fix the boundary is not left-orderable. He constructed a subgroup $H$ generated by six elements $a,b,c_0,d_0, c_1, d_1$, corresponding to $\alpha^{-1},\beta^{-1},\gamma, \delta, \gamma^{\eta}, \delta^{\eta}$ on \cite[page 4]{hyde2019group}. Let the subgroups $H_0, H_1$ be generated by $a,b,c_0,d_0$ and $a,b,c_1,d_1$ in $H$. Then he proved that, for any left order on $H_0$ (resp. $H_1$), we have $\max(a,a^{-1})<\max(b,b^{-1})$ (resp. $\max(a,a^{-1})>\max(b,b^{-1})$).
   
   In this example, we establish the non-left-orderability through automated reasoning. This proof is different from Hyde's proof and the one in \cite{triestino2021james}.
   
   The elements  $a,b,c_0,d_0, c_1, d_1$ satisfy the following relations: $ab=ba$, $c_0b=bc_0$, $d_0 b=b d_0$, $c_0 a^3 c_0 = a^3$, $d_0 a^3 d_0 = a^3$, $(d_0^{-1} c_0 d_0 a)^6 =a^6 b^{36}$, $c_1a=ac_1$, $d_1a=ad_1$, $c_1 b^3 c_1 = b^3$, $d_1 b^3 d_1= b^3$, $(d_1^{-1} c_1 d_1 b)^6 =a^{36} b^{6}$. Note that the terms containing $a^{36}$ or $b^{36}$ pose a significant computational challenge for the automated prover. To address this issue, our first task is to manually simplify these relations.

   Define the set of group elements $T=\{t_0,t_1,t_2,t_3,t_4,t_5,t_6\}$ in $H$ by $t_0=a^3$, $t_1=b^3$, $t_2=t_1^{-2} d_0^{-1} c_0 d_0 $, $t_3= t_0^{-2} d_1^{-1} c_1 d_1$, $t_4= t_2 a$, $t_5=t_3 b$, $t_6=t_4^3$, $t_7=t_5^3$. Let $R_0$ be the set of the following relations: $t_0=a^3$, $t_0 t_1=t_1 t_0$, $c_0 t_1 = t_1 c_0$, $d_0 t_1=t_1 d_0$, $c_0 t_0=t_0 c_0^{-1}$, $d_0 t_0=t_0d_0^{-1}$, $t_1^2 t_2 =d_0^{-1} c_0 d_0$, $t_4=t_2 a$, $t_6=t_4^3$, $t_6^2=t_0^2$. Similarly, let $R_1$ be the set of the following relations: $t_1=b^3$, $t_1 t_0=t_0 t_1$, $c_1 t_0=t_0 c_1$, $d_1 t_0=t_0 d_1$, $c_1 t_1=t_1 c_1^{-1}$, $d_1 t_1=t_1 d_1^{-1}$, $t_0^2 t_3=d_1^{-1} c_1 d_1$, $t_5= t_3 b$, $t_7=t_5^3$, $t_7^2=t_1^2$. Then we can check that $R_0$ and $R_1$ are satisfied in $H$. Note that the sets $R_0$ and $R_1$ are symmetric: $R_1$ can be obtained by replacing $a,c_0,d_0,t_0,t_1,t_2,t_4,t_6$ in $R_0$ with $b,c_1,d_1,t_1,t_0,t_3,t_5,t_7$ respectively.
   
   Let $G$ be the group with presentation \[\langle\{a,b,c_0,d_0,c_1,d_1\}\cup T\;|\;R_0\cup R_1\rangle,\] then $H$ is a quotient group of $G$. By the $\mathrm{Homeo}(D,\partial D)$-representation \cite{hyde2019group} of $H$, we conclude that $H$ is nontrivial. We are going to prove that $G$ does not admit nontrivial left-orderable quotients, which implies that $H$ is not left-orderable.

   For each $t\in\{a,c_0,d_0, t_1\}$, we check that $\Gr\cup\Ax_{R_0}\cup \CC \cup\{P(t_0)\land P(t_0'), \neg P(t)\}$ is inconsistent (Task 18.1, Task 18.2, Task 18.3 and Task 18.4) using Prover9. By Proposition~\ref{prop-cl}, we obtain that $a,c_0,d_0,t_1\in \mathrm{cl}(t_0)$ in $G$. By symmetry, we have $b,c_1,d_1,t_0\in \mathrm{cl}(t_1)$. Since $G$ is generated by $a,b,c_0,d_0,c_1,d_1$, we have $\mathrm{cl}(t_0)=\mathrm{cl}(t_1)=G$. In other words, $\langle t_0\rangle$ and $\langle t_1 \rangle$ are left absolutely cofinal in $G$. 

   We verify that the theories $\Gr\cup\Ax_{R_0}\cup \CC \cup\{P(t_0\cdot t_1)\land P((t_0\cdot t_1)'), \neg P(t_1)\}$ and $\Gr\cup\Ax_{R_0}\cup \CC \cup\{P(t_0\cdot t_1')\land P((t_0\cdot t_1')'), \neg P(t_1)\}$ are inconsistent (Task 18.5 and Task 18.6) using Prover9. By Proposition~\ref{prop-cl}, we obtain $t_1 \in \mathrm{cl}(t_0 t_1)$ and $t_1\in \mathrm{cl}(t_0t_1^{-1})$, which implies that $\langle t_0 t_1\rangle$ and $\langle t_0 t_1^{-1}\rangle$ are left absolutely cofinal in $G$.

   Consider the set $\mathfrak{S}^P$ with respect to the pairs $(e,t_0)$, $(e, t_1)$, $(e, t_0\cdot t_1)$ and $(e, t_0\cdot t_1')$. In order to apply Proposition~\ref{prop-fixed-point}, we would like to prove the inconsistency of $\Gr\cup \Ax_{R_0}\cup \Ax_{R_1}\cup \AxPC'\cup \PB\cup \Isolated(\{2\})\cup \mathfrak{S}^P$. However, this task is computationally challenging for Prover9, so we break the task into two cases according to the comparison between $\max(a,a^{-1})$ and $\max(b,b^{-1})$. 

   We verify that $\Gr\cup\Ax_{R_0}\cup\AxPC'\cup \PB\cup\Isolated(\{2\})\cup \mathfrak{S}^P\cup\{ P(t_0'\cdot t_1')\lor P(t_0'\cdot t_1)\}$ and $\Gr\cup\Ax_{R_1}\cup\AxPC'\cup \PB\cup\Isolated(\{2\})\cup \mathfrak{S}^P\cup\{\neg(P(t_0'\cdot t_1')\lor P(t_0'\cdot t_1))\}$ are inconsistent (Task 18.7 and Task 18.8) by Prover9. Since either $P(t_0'\cdot t_1')\lor P(t_0'\cdot t_1)$ or $\neg(P(t_0'\cdot t_1')\lor P(t_0'\cdot t_1))$ holds, the theory $\Gr\cup \Ax_{R_0}\cup \Ax_{R_1}\cup \AxPC'\cup \PB\cup \Isolated(\{2\})\cup \mathfrak{S}^P$ is inconsistent. By Proposition~\ref{prop-fixed-point}, the group $G$ does not admit nontrivial left-orderable quotients.
\end{example}

\section{Automated reasoning tasks}\label{sec:tasks}
We have applied our methodology to many groups given by their presentations. In these examples, numerous automated reasoning tasks are performed either by Prover9 to derive a contradiction or by Mace4 to find a finite model. We always use Knuth-Bendix ordering when performing Prover9 tasks. 

We provide a Python 3 script on \cite{zenodo_dataset} that generates input files for Prover9 and Mace4, and runs them by calling the automated theorem provers. This script comes with following predefined axiom sets in Prover9/Mace4 format: $\Gr$, $\AxL$, $\OrdL$, $\OrdB$, $\AxC$, $\OrdC$, $\OrdCB$, $\AxPC$, $\PB$, $\AxPCL$, $\PCB$, $\CC$, and $\Isolated(M)$. It also generates axiom sets $\Ax_R$, $\mathfrak{S}$, and $\mathfrak{S}^P$ from given pairs or triples.

For each automated reasoning task described in this paper, we create a task specifying the program name (\texttt{prover9} or \texttt{mace4}) and the selected axiom sets, and then execute this task. To run the Python program, one need to install Prover9 and Mace4 \cite{prover9-mace4}, and set the value of \texttt{BIN\_LOCATION} to the location of the binary files. The output of this program consists of the standard inputs and outputs of the automated theorem provers, which can also be accessed on \cite{zenodo_dataset}.

We execute our code using Python 3.9.12 and version LADR-Dec-2007 of Prover9 and Mace4, running on hardware with an AMD Ryzen 7 4800HS processor operating at 2.90 GHz, equipped with 16.0 GB of RAM, and operating on Windows 11 Home. Table~\ref{tab:summary1} summarises the time spent by Prover9 in seconds. 

\begin{table}[!ht]
 \centering
\begin{tabular}{||c||c|c|c||}
\hline 
\textbf{Prover9 tasks} & \textbf{User CPU time} & \textbf{System CPU time}  &  \textbf{Wall clock time}\\
\hline 
     Task 1.2 & 0.00 & 0.00 & 0   \\
\hline 
     Task 2.2 & 0.00 & 0.00 &  0 \\
\hline 
    Task 3.2  & 4.83  & 0.05  &  31\\
\hline 
    Task 3.3  & 1.78  & 0.08  &  31\\
    \hline
    Task 4.2 & 27.77 & 0.16 & 38\\
    \hline
    Task 5.2 & 973.94&16.66&1756\\
    \hline
    Task 6.1& 39.84&0.06& 44\\
    \hline
    Task 7.2& 0.00&0.00&0\\
    \hline
    Task 8.2& 0.00& 0.00& 0\\
    \hline
    Task 9.2& 15.31&0.00&19\\
    \hline
    Task 9.4& 25.72&0.03&32\\
    \hline
    Task 9.6&0.80&0.00&10\\
    \hline
    Task 9.8&6.73&0.05&40\\
    \hline
    Task 10.2& 59.33& 0.05&66\\
    \hline
    Task 10.4& 0.00&0.03&0\\
    \hline
    Task 11.1& 10.42&0.01&13\\
    \hline
    Task 11.2& 3.41& 0.00& 7\\
    \hline
    Task 12.1& 0.00& 0.00&0\\
    \hline
    Task 15.1& 0.00&0.00&1\\
    \hline
    Task 16.1&2.06&0.00&5\\
    \hline
    Task 16.2&0.00& 0.00&1\\
    \hline
    Task 17.1& 6.61&0.08&21\\
    \hline
    Task 17.2& 5.34&0.00&20\\
    \hline
    Task 17.3& 0.75&0.01&9\\
    \hline
    Task 18.1&0.00&0.00&0\\
    \hline
    Task 18.2& 0.00&0.00&0\\
    \hline
    Task 18.3& 0.00&0.00&0\\
    \hline
    Task 18.4& 0.16& 0.00&2\\
    \hline
    Task 18.5& 1.81&0.06&12\\
    \hline
    Task 18.6&0.64&0.00&8\\
    \hline
    Task 18.7&71.31&4.23&245\\
    \hline
    Task 18.8& 5.33&0.03&18\\
\hline
\end{tabular}
    \vspace*{3mm}
    \caption{Time spent by Prover9, measured in seconds.}
    \label{tab:summary1}
\end{table}

All of our Mace4 tasks can be solved in less than one second, except for Task 8.1, which takes 139 seconds to solve. Table~\ref{tab:summary2} summarises the sizes of the finite models found by Mace4.

\begin{table}[!ht]
    \centering
    \begin{tabular}{||c|c||c|c||}
    \hline
        \textbf{Mace4 tasks} & \textbf{model size} &\textbf{Mace4 tasks} &  \textbf{model size}\\
        \hline
     Task 1.1& 2&Task 9.1&14\\
     \hline
     Task 2.1& 2&Task 9.3&14\\
     \hline
     Task 3.1& 4 &Task 9.5&14\\
     \hline
     Task 4.1&6&Task 9.7&14\\
     \hline
     Task 5.1&14&Task 10.1&5\\
     \hline
     Task 7.1&6&Task 10.3&5\\
     \hline
     Task 8.1&60&Task 17.4&5\\
    
        \hline
    \end{tabular}
    \vspace*{3mm}
    \caption{Sizes of finite models found by Mace4.}
    \label{tab:summary2}
\end{table}

\subsubsection*{Acknowledgement}
The work of first and third named authors was supported by the Leverhulme Trust Research Project Grant RPG-2019-313. Part of the work was done when the second author visited Institut des Hautes Études Scientifiques.

\bibliographystyle{alpha}
\bibliography{ref}
\end{document}